\documentclass[12pt,a4paper]{amsart}
\usepackage[top=25mm, bottom=25mm, left=25mm, right=25mm]{geometry}
\usepackage[T1]{fontenc}
\usepackage{mathptmx}
\usepackage{amsmath,amssymb,mathtools,mathrsfs}
\usepackage{xcolor}
\usepackage{enumitem}
\usepackage[colorlinks=true,citecolor=blue]{hyperref}

\allowdisplaybreaks[3]
\setlist[enumerate]{leftmargin=*}
\newcommand{\UClim}{\operatorname*{UC\!-\!lim}}

\newtheorem{maintheorem}{Theorem}

\newtheorem{theorem}{Theorem}[section]
\newtheorem{proposition}[theorem]{Proposition}
\newtheorem{lemma}[theorem]{Lemma}
\newtheorem{corollary}[theorem]{Corollary}
\theoremstyle{definition}

\theoremstyle{remark}
\newtheorem{remark}[theorem]{Remark}
\numberwithin{equation}{section}

\newtheorem*{ac}{Acknowledgements}

\title[Polynomial diagonal orbit closures]
{Connectedness of polynomial diagonal orbit closures for minimal nilrotations and applications}

\author[K.~Ouyang]{Kangbo Ouyang}
\email{oy19981231@mail.ustc.edu.cn}

\author[J.~Qiu]{Jiahao Qiu}
\email{qiujh@mail.ustc.edu.cn}

\author[X.~Ye]{Xiangdong Ye}
\email{yexd@ustc.edu.cn}

\address[K.~Ouyang, J.~Qiu, X.~Ye]
{School of Mathematical Sciences, University of Science and Technology of
China, Hefei, Anhui 230026, P.R. China}

\date{\today}

\begin{document}

\begin{abstract}
For a minimal nilrotation on a compact connected nilmanifold, 
we prove that the polynomial diagonal orbit closure associated 
with any finite family of polynomials with integer coefficients 
vanishing at the origin is connected. This resolves a conjecture 
of Glasscock, Koutsogiannis, Le, Moreira, Richter, and Robertson. 
Combined with their equivalence theorem, our result yields 
polynomial multiple recurrence in every prescribed residue class 
in topological dynamics, provided that the corresponding power 
of the transformation is minimal. Furthermore, we independently 
establish the measure-theoretic counterpart of this recurrence 
phenomenon. Finally, we construct a totally minimal nilsystem 
for which the lower central series identity proposed by Leibman  fails.
\end{abstract}

\maketitle

\section{Introduction}

Multiple recurrence connects topological dynamics, ergodic theory, and 
additive combinatorics. Its classical origins lie in Szemer\'edi's 
theorem on arithmetic progressions \cite{Szemeredi1975} and Furstenberg's 
ergodic-theoretic proof \cite{Furstenberg1977,Furstenberg1981}, which 
interprets a finite combinatorial configuration as the simultaneous 
return of a measurable set under several iterates of a single 
transformation. Furstenberg and Weiss established the topological 
counterpart \cite{FurstenbergWeiss1978}, thereby obtaining the 
corresponding dynamical form of van der Waerden's theorem. These 
developments led naturally to the study of multiple recurrence under 
additional arithmetic restrictions on the recurrence time. In particular, 
Bergelson and Leibman proved polynomial extensions of both the 
measure-theoretic and topological multiple recurrence theorems for 
polynomials with integer coefficients vanishing at the origin 
\cite{BergelsonLeibman1996}.

One of the most basic arithmetic restrictions is the requirement that 
the recurrence time belong to a prescribed residue class. To describe 
the problem, let $(X,\mathcal B,\mu,T)$ be an invertible 
measure-preserving system, let $A\in\mathcal B$ have positive measure, 
and let $d\geq1$. Ordinary linear multiple recurrence yields a nonzero 
$n\in\mathbb Z$ such that
\[
\mu\bigl(
A\cap T^{-n}A\cap\cdots\cap T^{-dn}A
\bigr)>0.
\]

Given $k\in\mathbb N$ and $0\leq j<k$, one may ask whether such an $n$ 
can always be chosen from the progression $k\mathbb Z+j$, under the 
natural assumption that $T^k$ is ergodic. Host and Kra established 
this result for $d=3$ \cite{HostKra2002}, and Frantzikinakis 
extended it to arbitrary $d$ \cite{Frantzikinakis2004}. 
Glasner, Huang, Shao, Weiss, and Ye later established the corresponding 
topological theorem and formulated its polynomial extension 
\cite{GHSWY2025}.

The polynomial extension is not a formal consequence of ordinary 
polynomial multiple recurrence. Let $p_1,\ldots,p_d\in\mathbb Z[t]$ 
satisfy $p_i(0)=0$ for $1\leq i\leq d$. If the recurrence time is 
restricted by writing $n=kq+j$, then
\[
p_i(kq+j)
=
kQ_{i,j}(q)+p_i(j),
\qquad
Q_{i,j}(q)
:=
\frac{p_i(kq+j)-p_i(j)}{k}
\in\mathbb Z[q]
\qquad (1\leq i\leq d).
\]
Although $Q_{i,j}(0)=0$, the constants $p_i(j)$ shift the target sets
appearing in the multiple intersection. Consequently, restricting the
recurrence time to a residue class does not reduce directly to ordinary
polynomial multiple recurrence for $T^k$ with the family
$(Q_{i,j})_{i=1}^d$. The problem instead requires a comparison of
polynomial recurrence across the congruence classes while retaining
control of these shifted targets.

Nilsystems arise naturally in this problem through characteristic 
factor theory. In the measure-theoretic setting, Host and Kra 
\cite{HostKra2005Nil} and Ziegler \cite{Ziegler2007} showed that 
suitable pro-nilfactors are characteristic for linear multiple 
averages. The corresponding polynomial results were obtained by 
Host and Kra \cite{HostKra2005Polynomial} and by Leibman 
\cite{Leibman2005Averages}. In topological dynamics, Host, Kra, 
and Maass \cite{HostKraMaass2010} and Shao and Ye \cite{ShaoYe2012} 
developed higher-order regionally proximal relations and the 
associated pro-nilfactors. Glasner, Huang, Shao, Weiss, and Ye 
\cite{GHSWY2025} then identified the maximal infinite-step 
pro-nilfactor as a characteristic factor for linear configurations. 
Polynomial characteristic factor results were subsequently obtained by Qiu \cite{Qiu2023} 
and by Huang, Shao, and Ye \cite{HuangShaoYe2025}, while Ye and Yu 
\cite{YeYu2025} showed that a finite-step pro-nilfactor suffices 
for fixed polynomial family.

Using this theory, Glasscock, Koutsogiannis, Le, Moreira, Richter, and
Robertson established a return set comparison for essentially distinct
polynomial families in minimal systems. The factor return set,
defined using the interiors of the factor images, contains the original
return set, while the additional return times in the factor do not form a
piecewise syndetic set
\cite{GlasscockEtAl2026}.
They further showed that polynomial recurrence in prescribed residue classes is equivalent 
to connectedness of the associated polynomial diagonal orbit closures in 
totally minimal nilsystems. Consequently, 
the recurrence problem is reduced to this concrete connectedness question for 
polynomial orbits on nilmanifolds.

\subsection{Connectedness and recurrence statements}

A \emph{nilmanifold} is a compact homogeneous space $G/\Gamma$, where $G$ is 
a finite-dimensional nilpotent Lie group and $\Gamma\leq G$ is a discrete 
cocompact subgroup. For $a\in G$, the map $T_a(g\Gamma)=ag\Gamma$ is a 
\emph{nilrotation}, and $(G/\Gamma,T_a)$ is a \emph{nilsystem}. The group $G$ 
in a chosen presentation need not be connected, even when $G/\Gamma$ is 
connected. 

Let $m,d\in\mathbb N$, let $(X,T)$ be a nilsystem, and let
\[
\mathbf p=(p_1,\ldots,p_d)
\in\mathbb Z[t_1,\ldots,t_m]^d,
\qquad
\Delta_X^d:=\{(x,\ldots,x):x\in X\}.
\]
We write
\begin{equation}\label{eq:intro-orbit-closure}
\overline{\mathcal O_{\mathbf p}(\Delta_X^d)}
:=
\overline{\bigl\{
\bigl(T^{p_1(\mathbf n)}x,\ldots,T^{p_d(\mathbf n)}x\bigr):
x\in X,\ \mathbf n\in\mathbb Z^m
\bigr\}}
\subseteq X^d.
\end{equation}

\noindent\textbf{Convention.}
Unless stated otherwise, all polynomials are nonzero.

For the recurrence and connectedness statements below, this convention
entails no loss.
Zero polynomials are redundant in the recurrence statements, while
diagonal orbit closures are invariant under common polynomial shifts:
for every
$h\in\mathbb Z[t_1,\ldots,t_m]$ with $h(\mathbf 0)=0$,
\[
\overline{\mathcal O_{(p_1+h,\ldots,p_d+h)}(\Delta_X^d)}
=
\overline{\mathcal O_{\mathbf p}(\Delta_X^d)},
\]
since, for each fixed $\mathbf n$, the map
$x\mapsto T^{h(\mathbf n)}x$ is a bijection of $X$.
For all but finitely many $M\in\mathbb Z$, taking
$h(\mathbf t)=Mt_1$ makes every $p_i+h$ nonzero.

The use of the diagonal reflects the fact that recurrence concerns the
simultaneous motion of a single initial point along several polynomial times.

Leibman's polynomial mapping framework was developed in 
\cite{Leibman1998,Leibman2002}. Results describing the closures and 
distributions of point orbits are given in \cite{Leibman2005,Leibman2005Zd}, 
while rationality and the orbits of moving subnilmanifolds are treated in 
\cite{Leibman2006Rational,Leibman2007}. Green and Tao developed quantitative 
factorization and equidistribution results for polynomial nilorbits 
\cite{GreenTao2012}. For polynomial diagonal orbits, Leibman obtained 
explicit descriptions for linear families, for presentations by connected 
groups, and for presentations in which the identity component of the 
presenting group is abelian \cite{Leibman2010}.
Outside these explicit cases, the general orbit theory yields a finite
decomposition into connected subnilmanifolds but does not rule out the
presence of more than one component \cite{Leibman2010}.
The issue addressed in this paper is whether this finite decomposition collapses to a single component when the nilmanifold is connected, 
the nilrotation is minimal, and $p_i(\mathbf 0)=0$.

We distinguish the following three statements.
\begin{description}[
style=multiline,
leftmargin=1.75cm,
labelwidth=1.45cm,
labelsep=0.15cm,
itemsep=0.75em,
topsep=0.5em,
parsep=0pt,
partopsep=0pt,
font=\normalfont\bfseries
]

\item[\textup{(Top)}]
Let $(X,T)$ be a topological dynamical system, let $k,d\in\mathbb N$, and 
suppose that $(X,T^k)$ is minimal. Then, for every nonempty open set 
$U\subseteq X$, every $\mathbf p=(p_1,\ldots,p_d)\in\mathbb Z[t]^d$ 
satisfying $p_i(0)=0$ for $1\leq i\leq d$, and every $0\leq j<k$, there 
exists a nonzero $n\in\mathbb Z$ with $n\equiv j\pmod{k}$ such that
\[
U\cap T^{-p_1(n)}U\cap\cdots\cap T^{-p_d(n)}U\neq\emptyset.
\]

\item[\textup{(Conn)}]
Let $(X,T)$ be a connected minimal nilsystem, and let $d\in\mathbb N$. 
Then, for every $\mathbf p=(p_1,\ldots,p_d)\in\mathbb Z[t]^d$ satisfying 
$p_i(0)=0$ for $1\leq i\leq d$, 
$\overline{\mathcal O_{\mathbf p}(\Delta_X^d)}$ is connected.

\item[\textup{(Meas)}]
Let $(X,\mathcal B,\mu,T)$ be an invertible measure-preserving system, let $k,d\in\mathbb N$, and suppose that $T^k$ 
is ergodic. Then, for every $A\in\mathcal B$ with $\mu(A)>0$, every 
$\mathbf p=(p_1,\ldots,p_d)\in\mathbb Z[t]^d$ satisfying $p_i(0)=0$ for 
$1\leq i\leq d$, and every $0\leq j<k$, there exists a nonzero $n\in\mathbb Z$ 
with $n\equiv j\pmod{k}$ such that
\[
\mu\bigl(
A\cap T^{-p_1(n)}A\cap\cdots\cap T^{-p_d(n)}A
\bigr)>0.
\]
\end{description}

By the common shift observation above, \textup{(Conn)} is equivalent to
\cite[Conjecture~1.5]{GlasscockEtAl2026}, which is presented there as a
formal special case of \cite[Conjecture~11.4]{Leibman2010}.
By \cite[Theorem~4.23]{GlasscockEtAl2026}, \textup{(Conn)} is equivalent
to the syndetic form of \textup{(Top)}. We use only the implication from
\textup{(Conn)} to \textup{(Top)}; \textup{(Meas)} is proved separately.

\subsection{Main results}

\begin{maintheorem}
\label{thm:main}
Let $m,d\in\mathbb N$, and let $X=G/\Gamma$ be a connected nilmanifold, where 
$G$ is a nilpotent Lie group and $\Gamma\leq G$ is a discrete cocompact 
subgroup. Let $a\in G$, and suppose that the nilrotation $T=T_a$ is minimal. 
If $\mathbf p=(p_1,\ldots,p_d)\in\mathbb Z[t_1,\ldots,t_m]^d$ satisfies 
$p_i(\mathbf 0)=0$ for $1\leq i\leq d$, then 
$\overline{\mathcal O_{\mathbf p}(\Delta_X^d)}$ is connected.
\end{maintheorem}

Repetitions are allowed. By the standing convention, each $p_i$ is nonzero
and hence nonconstant. The presenting group $G$ may be disconnected; passage
to its identity component generally turns the nilrotation into a unipotent
affine transformation.

Subtracting the constant terms changes the orbit closure by the
coordinatewise homeomorphism
\[
(x_1,\ldots,x_d)
\longmapsto
\bigl(T^{-p_1(\mathbf 0)}x_1,\ldots,T^{-p_d(\mathbf 0)}x_d\bigr).
\]
Together with the common shift observation above, this shows that the
normalization $p_i(\mathbf 0)=0$ entails no loss.

The multiparameter statement follows from the one-parameter case. For
$\mathbf v\in\mathbb Z^m$, set
$q_{i,\mathbf v}(t):=p_i(t\mathbf v)$.
Choose $M_{\mathbf v}\in\mathbb Z$ so that every
$q_{i,\mathbf v}(t)+M_{\mathbf v}t$ is nonzero. The one-parameter theorem
and the common shift identity show that every corresponding line orbit
closure is connected. All these closures contain $\Delta_X^d$ at $t=0$,
so their union is connected. Every multiparameter orbit point belongs to
this union by taking $\mathbf v=\mathbf n$ and $t=1$, while each line orbit
closure is contained in the full orbit closure. Hence the closure of the
union is $\overline{\mathcal O_{\mathbf p}(\Delta_X^d)}$, which is
connected.

By Proposition~\ref{prop:minimal-total-minimal},
Theorem~\ref{thm:main} proves \textup{(Conn)} and hence \textup{(Top)}.
The stronger syndetic conclusion is recorded in
Corollary~\ref{cor:return-times}.
We next establish a stronger measure-theoretic result.

\begin{maintheorem}
\label{thm:measure-common-limit}
Let $(X,\mathcal B,\mu,T)$ be an invertible measure-preserving system, let $k,d\in\mathbb N$, and suppose that $T^k$ is 
ergodic. Let $A\in\mathcal B$, and let $\mathbf p=(p_1,\ldots,p_d)\in\mathbb Z[t]^d$ 
satisfy $p_i(0)=0$ for $1\leq i\leq d$. For $n\in\mathbb Z$, define
\begin{equation}\label{eq:intro-correlation}
c(n)
:=
\mu\bigl(
A\cap T^{-p_1(n)}A\cap\cdots\cap T^{-p_d(n)}A
\bigr).
\end{equation}
Then there exists a constant $C\geq0$ such that the first limit below 
holds and, for each fixed $j\in\{0,1,\ldots,k-1\}$, the second limit holds.
\begin{equation}\label{eq:measure-moving-intervals}
\begin{aligned}
&\lim_{N\to\infty}\sup_{M\in\mathbb Z}
\left|
\frac1N\sum_{n=M}^{M+N-1}c(n)-C
\right|=0,\\
&\lim_{N\to\infty}\sup_{M\in\mathbb Z}
\left|
\frac1N\sum_{q=M}^{M+N-1}c(kq+j)-C
\right|=0.
\end{aligned}
\end{equation}
\end{maintheorem}

\begin{corollary}
\label{cor:measure-positive-threshold}
Under the hypotheses of Theorem~\ref{thm:measure-common-limit}, assume that 
$\mu(A)>0$. Then the constant $C$ in \eqref{eq:measure-moving-intervals} is 
strictly positive. Moreover, for every $0<\varepsilon<C$ and every $0\leq j<k$, 
the set
\[
\bigl\{
q\in\mathbb Z:
\mu\bigl(
A\cap T^{-p_1(kq+j)}A\cap\cdots\cap T^{-p_d(kq+j)}A
\bigr)>\varepsilon
\bigr\}
\]
is syndetic.
\end{corollary}

For each fixed $j$, the syndetic superlevel set corresponding to
$\varepsilon=C/2$ is infinite, so it contains some $q$ with $kq+j\neq0$. This proves \textup{(Meas)}.
Since the intersections in Corollary~\ref{cor:measure-positive-threshold} contain an additional
unshifted copy of $A$, the theorem and corollary also prove the invertible case of Conjecture~1.6 in \cite{GlasscockEtAl2026}.

\subsection{A counterexample to Leibman's conjectured lower central series identity}

The next result, which is not used in the proof of
Theorem~\ref{thm:main}, disproves the lower central series identity proposed
in \cite[Conjecture~11.4]{Leibman2010}.

Let $(X,T_a)$ be a connected nilsystem with $X=G/\Gamma$, where $T_a$ is 
minimal and $G=\langle G^0,a\rangle$. Put $N:=G^0$. Since $G/\Gamma$ is 
connected, one has $G=N\Gamma$. Choose $\alpha_0\in N$ and $\gamma\in\Gamma$ 
such that $a=\alpha_0\gamma^{-1}$. For this fixed $\gamma$, define
\[
g_\alpha(r):=(\alpha\gamma^{-1})^r\gamma^r\in N
\qquad
(\alpha\in N,\ r\in\mathbb Z).
\]
Given $\mathbf p=(p_1,\ldots,p_d)\in\mathbb Z[t_1,\ldots,t_m]^d$ with 
$p_i(\mathbf 0)=0$ for $1\leq i\leq d$, define
\begin{equation}\label{eq:intro-Hhat}
\widehat H
:=
\bigl\langle
\Delta_N^d,
\bigl(g_\alpha(p_1(\mathbf n)),\ldots,
g_\alpha(p_d(\mathbf n))\bigr)
\colon
\alpha\in N,\ \mathbf n\in\mathbb Z^m
\bigr\rangle
\leq N^d.
\end{equation}
For a group $K$, let $K_1:=K$ and $K_{r+1}:=[K,K_r]$. In this notation, 
\cite[Conjecture~11.4]{Leibman2010} asserts that
\begin{equation}\label{eq:intro-lower-central-identity}
\widehat H_r=\widehat H\cap N_r^d
\qquad (r\geq1).
\end{equation}
Here the angle brackets denote the abstract subgroup generated by the
indicated elements; no topological closure is taken.
Our notation reverses Leibman's symbols: his $r$ denotes the number of 
coordinates and his $d$ the lower central series index, whereas these roles 
are played here by $d$ and $r$, respectively.

The inclusion
$\widehat H_r\subseteq\widehat H\cap N_r^d$ is automatic, but the theorem
below shows that the reverse inclusion can fail for a connected
nilmanifold presented by a disconnected group.

\begin{maintheorem}
\label{thm:counterexample}
There exist a three-step nilpotent Lie group $G$, a lattice $\Gamma\leq G$, 
and elements
\[
\alpha_0\in N:=G^0,
\qquad
\gamma\in\Gamma,
\qquad
a:=\alpha_0\gamma^{-1}\in G,
\]
such that $G=\langle N,a\rangle$, the nilmanifold $G/\Gamma$ is connected, 
the nilrotation $T_a$ on $G/\Gamma$ is totally minimal, and 
$N\cong H_3(\mathbb R)\times\mathbb R^2$. For the four linear polynomials 
$p_i(n)=i\;n$, $1\leq i\leq4$, let $\widehat H$ be the subgroup defined in 
\eqref{eq:intro-Hhat} using the fixed element $\gamma$ above. Then
\[
\widehat H_2
\subsetneq
\widehat H\cap N_2^4,
\qquad
(\widehat H\cap N_2^4)/\widehat H_2
\cong
(\mathbb R,+)
\]
as topological groups. In particular, 
\eqref{eq:intro-lower-central-identity} fails at $r=2$.
\end{maintheorem}

Since the quotient in Theorem~\ref{thm:counterexample} is isomorphic to 
$(\mathbb R,+)$, the failure is not a finite-index phenomenon. Nevertheless, 
the example satisfies the hypotheses of Theorem~\ref{thm:main}, and its 
polynomial diagonal orbit closure is connected. Thus connectedness of the 
orbit closure does not imply the lower central series identity.

\subsection{Proof strategy and organization}
The main difficulty in proving Theorem~\ref{thm:main} is to show that the
finite component decomposition supplied by polynomial orbit theory consists
only of the component containing the diagonal. We address the difficulties
caused by disconnected presentations and discrete polynomial parameters
through an affine reduction, interpolation in real parameters, and
induction on the nilpotency step.

Since $G/\Gamma$ is connected, one has $G=G^0\Gamma$. After decomposing the 
translation element and passing to a simply connected cover of $G^0$, we 
reduce the system to a minimal affine transformation
$F(x\Lambda)=bA(x)\Lambda$,
where $N$ is connected and simply connected, $\Lambda$ is a lattice, and $A$ 
is a unipotent automorphism preserving $\Lambda$. A nilpotent suspension 
realizes this affine system as a nilrotation on a larger nilmanifold 
\cite[Section~4]{DaniShahSharma2015}.
We then construct a curve that agrees with the affine iterates at every
integer time. Composing this curve with the polynomials $p_i$ gives curves
whose logarithmic derivatives provide the polynomial tangent data needed
for the component argument.

Since $p_i(\mathbf 0)=0$, the diagonal $\Delta_X^d$ is contained in
$\overline{\mathcal O_{\mathbf p}(\Delta_X^d)}$. Let $C_0$ denote the
connected component containing $\Delta_X^d$. We prove that 
the polynomial diagonal orbit is contained in $C_0$ by induction on the 
nilpotency step. After quotienting by the last nontrivial term of the lower 
central series, the induction hypothesis reduces the remaining obstruction to 
a compact connected torus. A character of this torus converts the 
interpolation into a real polynomial phase. The finite component decomposition 
imposes rational constraints on this phase at integer parameters, whereas 
minimality forces the corresponding nonzero invariant horizontal functional 
to take an irrational value on the translation direction. 
This contradiction eliminates the residual torus obstruction and
completes the induction.

For Theorem~\ref{thm:measure-common-limit}, we choose a single pro-nilfactor 
that is simultaneously characteristic  for the original polynomial 
family and for the finitely many reparameterizations associated with the 
residue classes. At a finite nilsystem stage with $\kappa$ connected components, the
$k$th power of the induced nilrotation is ergodic, since it is a factor
of $T^k$. Hence $\gcd(k,\kappa)=1$. Theorem~\ref{thm:main} and 
Haar distribution determine the limit on each component, while multiplication 
by $k$ permutes the component classes. The full sequence and each fixed 
subsequence $q\mapsto c(kq+j)$ therefore have the same uniform Ces\`aro limit. 
Uniform $L^1$ stability and characteristic factor estimates then transfer
this conclusion first to the inverse limit and then to the original system.

The counterexample is constructed from an explicit three-step semidirect
product in which a cubic slope direction belongs to
$\widehat H\cap N_2^4$ but not to $\widehat H_2$.

Section~\ref{sec:preliminaries} introduces the required conventions.
Sections~\ref{sec:affine-suspension}--\ref{sec:induction} develop the
affine reduction and prove Theorem~\ref{thm:main}.
Section~\ref{sec:applications} gives its recurrence and combinatorial
consequences. 
Section~\ref{sec:counterexample} proves
Theorem~\ref{thm:counterexample}, while
Appendix~\ref{sec:two-step-model} presents a two-step model of the
central induction.

\begin{ac}
The authors thank Professors Wen Huang and Song Shao for many helpful 
discussions. This research was supported by the National Key Research 
and Development Program of China (Nos.~2024YFA1013601 and 2024YFA1013600) 
and the National Natural Science Foundation of China (Nos.~12031019, 
12201599, 12401243, 12426201, and 12471188).
\end{ac}

\section{Preliminaries}\label{sec:preliminaries}

This section fixes notation and recalls the background used below.
We write $\mathbb N=\{1,2,\ldots\}$ and use $e$ for the identity element
of a group. All Lie groups are finite-dimensional, and
$\operatorname{Aut}(G)$ denotes the automorphism group of $G$.
Homogeneous spaces are written as $G/\Gamma$, with $G$ acting by left
translation.
For a compact abelian group $K$, we write $\widehat K$ for its group of
continuous homomorphisms into $\mathbb R/\mathbb Z$, written additively.

\subsection{Dynamical systems}

A \emph{topological dynamical system} is a pair $(X,T)$, where $X$ is
a compact metric space and $T:X\to X$ is a homeomorphism. 
A continuous surjection $\pi:X\to Y$ satisfying
$\pi\circ T=S\circ\pi$ is a \emph{topological factor map} from $(X,T)$
to $(Y,S)$. If $\pi$ is a homeomorphism, the two systems are
\emph{topologically conjugate}.

The system $(X,T)$ is \emph{transitive} if it has a point with a dense
orbit, and it is \emph{minimal} if every orbit is dense. Equivalently,
minimality means that $X$ has no proper nonempty closed $T$-invariant
subset. It is \emph{totally minimal} if $(X,T^k)$ is minimal for every $k\in\mathbb N$.

A probability space $(X,\mathcal B,\mu)$ is called \emph{standard} if it is
isomorphic modulo null sets to a Borel probability space on a Polish space.
A \emph{measure-preserving system} is a quadruple
$(X,\mathcal B,\mu,T)$, where $(X,\mathcal B,\mu)$ is a standard
probability space and $T:X\to X$ is an invertible measure-preserving
transformation. A measurable factor may be described either by a factor map
modulo null sets or, equivalently, by a $T$-invariant
sub-$\sigma$-algebra of $\mathcal B$. The system is \emph{ergodic} if every
$T$-invariant measurable set has measure $0$ or $1$, and it is
\emph{totally ergodic} if $T^k$ is ergodic for every $k\in\mathbb N$.

\subsection{Nilmanifolds and nilsystems}

For background on nilmanifolds, nilsystems, and polynomial nilorbits, see
\cite{Parry1969,Raghunathan1972,Leibman2005,Leibman2005Zd,HostKra2018}.

\subsubsection{Nilmanifolds and subnilmanifolds}

Let $G$ be a Lie group with identity $e$.
Its lower central series is defined recursively by
$G_1:=G$ and $G_{j+1}:=[G,G_j]$.
Here, $[G,G_j]$ is the subgroup generated by the commutators
$[x,y]:=x^{-1}y^{-1}xy$ for $x\in G$ and $y\in G_j$.
A group is \emph{nilpotent} if $G_{s+1}=\{e\}$ for some $s\geq0$.
The least such $s$ is its nilpotency class; a nontrivial group of class
$s$ is called \emph{$s$-step nilpotent}.

A \emph{lattice} in a Lie group $G$ is a discrete subgroup
$\Gamma\leq G$ such that $G/\Gamma$ is compact; 
we write the points of $G/\Gamma$ as $g\Gamma$.

If $G$ is nilpotent and $\Gamma$ is a lattice in $G$, the compact
homogeneous space $X=G/\Gamma$ is called a \emph{nilmanifold}. 
The group $G$ acts continuously and
transitively on $X$ by left translation.
The presenting group $G$ may be disconnected even when $G/\Gamma$ is
connected.

Every nilmanifold carries a unique $G$-invariant Borel probability
measure, denoted by $m_{G/\Gamma}$ and called its \emph{Haar measure}.

 A closed subgroup $H\leq G$ is
\emph{rational with respect to $\Gamma$} if $H\cap\Gamma$ is a lattice in
$H$. Equivalently, $H\Gamma$ is closed in $G$; see
\cite{Leibman2006Rational,Raghunathan1972}.

For a closed subgroup $H\leq G$ and $g\in G$, the orbit
$Y:=Hg\Gamma/\Gamma$ is closed if and only if
$H\cap g\Gamma g^{-1}$ is a lattice in $H$, or equivalently, if
$g^{-1}Hg$ is rational with respect to $\Gamma$. Such a closed orbit is
called a \emph{subnilmanifold}, and its unique $H$-invariant probability measure is denoted by $m_Y$.
A \emph{connected subnilmanifold} is one admitting such a presentation with $H$ connected.

\subsubsection{Nilsystems}
For $a\in G$, let
\[
T_a:G/\Gamma\to G/\Gamma,
\qquad
T_a(g\Gamma):=ag\Gamma.
\]
The map $T_a$ is called a \emph{nilrotation}, and
$(G/\Gamma,T_a)$ is called a \emph{nilsystem}. Every nilrotation
preserves the Haar measure.

For a nilsystem $(X,T)$ with Haar measure $m_X$, transitivity and
minimality of $(X,T)$ are equivalent to ergodicity of
$(X,m_X,T)$. In this case, $m_X$ is the unique $T$-invariant Borel
probability measure
\cite[Theorem~4.1(1)]{BergelsonHostKra2005}.

We use the following characterization of total minimality; see
\cite{Parry1969,Leibman2005}.

\begin{proposition}\label{prop:minimal-total-minimal}
Let $(X,T)$ be a nilsystem, and let $m_X$ denote its Haar
measure. Then the following statements are equivalent:
\begin{enumerate}[label=\textup{(\roman*)}]
\item $(X,T)$ is totally minimal;
\item $(X,m_X,T)$ is totally ergodic;
\item $(X,T)$ is minimal and $X$ is connected.
\end{enumerate}
\end{proposition}

The connected components of a nilmanifold have the following standard
description; see \cite[Chapter~II]{Raghunathan1972}.

\begin{lemma}\label{lem:nilmanifold-components}
Let $X=G/\Gamma$ be a compact nilmanifold, and let $G^0$ be the identity
component of $G$. Then
\[
X_e:=G^0\Gamma/\Gamma
\cong
G^0/(G^0\cap\Gamma)
\]
is the connected component
of $X$
containing the identity coset. Every connected
component of $X$ is a left translate of $X_e$, and $X$ has only finitely
many connected components. In particular,
$X$ is connected if and only if $G=G^0\Gamma$.
\end{lemma}

We identify $(G/\Gamma)^d$ with $G^d/\Gamma^d$. For $H\leq G$, write
\[
\Delta_H^d:=\{(h,\ldots,h):h\in H\},
\qquad
\Delta_{G/\Gamma}^d
:=
\{(g\Gamma,\ldots,g\Gamma):g\in G\}.
\]
The latter is a closed subnilmanifold naturally homeomorphic to
$G/\Gamma$.

We shall repeatedly use the following universal cover construction; see
\cite{Raghunathan1972}.

\begin{lemma}\label{lem:universal-cover-package}
Let $G$ be a connected nilpotent Lie group, and let
$\pi:\widetilde G\to G$ be its universal covering homomorphism. Then
$\widetilde G$ is connected, simply connected, and nilpotent, and
$\ker\pi$ is discrete and central. If $\Gamma$ is a lattice in $G$, then
$\widetilde\Gamma:=\pi^{-1}(\Gamma)$ is a lattice in $\widetilde G$, and
$\pi$ induces a homeomorphism
\[
\widetilde G/\widetilde\Gamma\longrightarrow G/\Gamma.
\]
Every $A\in\operatorname{Aut}(G)$ has a unique lift
$\widetilde A\in\operatorname{Aut}(\widetilde G)$ satisfying
$\pi\circ\widetilde A=A\circ\pi$. If $A(\Gamma)=\Gamma$, then
$\widetilde A(\widetilde\Gamma)=\widetilde\Gamma$. Moreover,
$d_e\widetilde A$ is unipotent if and only if $d_eA$ is unipotent.
\end{lemma}

\subsection{Lie algebras and rational structures}

\subsubsection{Exponential coordinates and tangent conventions}

The lower central series of a Lie algebra
$\mathfrak g$ is
$\mathfrak g_1:=\mathfrak g$ and
$\mathfrak g_{j+1}:=[\mathfrak g,\mathfrak g_j]$.
If $G$ is a connected and simply connected nilpotent Lie group with Lie
algebra $\mathfrak g$, then the exponential map
$\exp:\mathfrak g\to G$ is a diffeomorphism, its inverse is denoted by
$\log$, and the Baker--Campbell--Hausdorff formula is a finite Lie
polynomial. For $g\in G$ and $r\in\mathbb R$, define
$g^r:=\exp(r\log g)$. The lower central series of $G$ and
$\mathfrak g$ are related by
$G_j=\exp(\mathfrak g_j)$ for $j\geq1$.
In particular, every $G_j$ is closed and connected, and the last nontrivial
term is central. We refer to \cite[Chapter~1]{CorwinGreenleaf1990} for
these facts.

For a Lie subgroup $H\leq G$, its Lie algebra is denoted by
$\operatorname{Lie}(H)$.  For $g\in G$, define
$\operatorname{Int}(g)(h):=ghg^{-1}$ and
$\operatorname{Ad}(g):=d_e\operatorname{Int}(g)$. For
$X\in\mathfrak g$, set $\operatorname{ad}(X)Y:=[X,Y]$.

Write $\lambda_g(h):=gh$ and $\rho_g(h):=hg$. For an interval
$I\subseteq\mathbb R$ and a differentiable curve $c:I\to G$, define
\[
\omega_L(c)(t)
:=
(d\lambda_{c(t)^{-1}})_{c(t)}c'(t),
\qquad
\omega_R(c)(t)
:=
(d\rho_{c(t)^{-1}})_{c(t)}c'(t).
\]
We also write
$c(t)^{-1}c'(t):=\omega_L(c)(t)$ and
$c'(t)c(t)^{-1}:=\omega_R(c)(t)$.

The following elementary identities record the
tangent-space conventions used throughout the proof.

\begin{lemma}\label{lem:tangent-trivialization}
Let $H\leq G$ be a Lie subgroup with Lie algebra $\mathfrak h$, and let
$y\in G$.
\begin{enumerate}[label=\textup{(\roman*)}]
\item Under the right trivialization at $y$, the tangent space
$T_y(Hy)$ is identified with $\mathfrak h$; under the left trivialization, it
is identified with $\operatorname{Ad}(y^{-1})\mathfrak h$.
\item If $Z\in\mathfrak g$ and $c_1(t):=y\exp(tZ)$, then
$\omega_L(c_1)(0)=Z$ and
$\omega_R(c_1)(0)=\operatorname{Ad}(y)Z$.
If $c_2(t):=\exp(tZ)y$, then
$\omega_L(c_2)(0)=\operatorname{Ad}(y^{-1})Z$ and
$\omega_R(c_2)(0)=Z$.
\end{enumerate}
If $\Gamma\leq G$ is a discrete subgroup and
$\pi_\Gamma:G\to G/\Gamma$ is the quotient map, then $\pi_\Gamma$ is a
local diffeomorphism, and the
differential $(d\pi_\Gamma)_y$ transports the preceding tangent-space
identities to $T_{y\Gamma}(G/\Gamma)$.
\end{lemma}

\begin{proof}
Since $\rho_y$ maps $H$ diffeomorphically onto $Hy$,
we have $T_y(Hy)=(d\rho_y)_e\mathfrak h$. Right trivialization therefore
identifies $T_y(Hy)$ with $\mathfrak h$. Since
$(d\lambda_{y^{-1}})_y(d\rho_y)_e=\operatorname{Ad}(y^{-1})$, left
trivialization identifies $T_y(Hy)$ with
$\operatorname{Ad}(y^{-1})\mathfrak h$. This proves \textup{(i)}. The
formulas in \textup{(ii)} follow directly from the definitions of
$\omega_L$ and $\omega_R$.
\end{proof}

For differentiable curves $c_1,c_2:I\to G$, their
pointwise product satisfies
\[
\omega_L(c_1c_2)
=
\operatorname{Ad}(c_2^{-1})\omega_L(c_1)+\omega_L(c_2).
\]
If $\varphi:G\to H$ is a Lie group homomorphism and
$c$ is a differentiable curve in $G$, then
\[
\omega_L(\varphi\circ c)
=
d_e\varphi\bigl(\omega_L(c)\bigr).
\]

\begin{lemma}\label{lem:zero-logarithmic-derivative}
Let $G$ be a Lie group, let $I\subseteq\mathbb R$ be an interval, and let
$c:I\to G$ be a continuously differentiable curve. If
\[
\omega_L(c)(t)=0
\qquad
(t\in I),
\]
then $c$ is constant. The same conclusion holds if the right logarithmic
derivative vanishes identically.
\end{lemma}

\begin{proof}
Both logarithmic trivializations are linear isomorphisms, so a vanishing
logarithmic derivative implies $c'(t)=0$ for every $t\in I$. Hence $c$
is constant on $I$.
\end{proof}

An automorphism $A\in\operatorname{Aut}(G)$ is \emph{unipotent} if
$A_*:=d_eA$ is a unipotent linear transformation, or equivalently, if
$(A_*-\operatorname{id}_{\mathfrak g})^r=0$ for some $r\in\mathbb N$.

\subsubsection{Mal'cev rational structures}
Let $G$ be a connected and simply connected nilpotent Lie group with Lie
algebra $\mathfrak g$, and let $\Gamma\leq G$ be a lattice. The associated
Mal'cev rational structure
on $\mathfrak g$
is
\[
\mathfrak g_{\mathbb Q}
:=
\operatorname{span}_{\mathbb Q}(\log\Gamma),
\qquad
\mathfrak g_{\mathbb Q}\otimes_{\mathbb Q}\mathbb R
=
\mathfrak g.
\]
A linear subspace $\mathfrak h\leq\mathfrak g$ is \emph{rational} if
\[
\mathfrak h
=
(\mathfrak h\cap\mathfrak g_{\mathbb Q})
\otimes_{\mathbb Q}\mathbb R.
\]
A connected Lie subgroup $H\leq G$ is \emph{rational} if
$\operatorname{Lie}(H)$ is rational. By
Proposition~\ref{prop:malcev-package}, this agrees with the preceding
group-theoretic definition.
Rationality is always understood relative to the lattice currently under
consideration. For the product presentation $(G^d,\Gamma^d)$, the rational
Lie algebra is $(\mathfrak g_{\mathbb Q})^d$. A linear map between Lie
algebras equipped with rational structures is called \emph{rational} if it
maps the rational structure of the domain into that of the codomain. In
particular, a functional $\ell\in\mathfrak g^*$ is rational if
$\ell(\mathfrak g_{\mathbb Q})\subseteq\mathbb Q$.

A \emph{full lattice} in a finite-dimensional real vector space $V$ is a
discrete subgroup $\Omega\leq V$ such that $V/\Omega$ is compact. Its dual
lattice is
\[
\Omega^*
:=
\{\ell\in V^*: \ell(\Omega)\subseteq\mathbb Z\}.
\]

We collect the Mal'cev rationality facts used below; see
\cite[Chapter~5]{CorwinGreenleaf1990}, \cite{Malcev1962}, and
\cite[Chapter~II]{Raghunathan1972}.

\begin{proposition}\label{prop:malcev-package}
Let $G$ be a connected and simply connected nilpotent Lie group with Lie
algebra $\mathfrak g$ and lattice $\Gamma$.
Then the following statements hold:
\begin{enumerate}[label=\textup{(\roman*)}]
\item $\mathfrak g_{\mathbb Q}$ is a Lie algebra over $\mathbb Q$, and every
term of the lower central series of $\mathfrak g$ is rational.
\item If $\mathfrak h\leq\mathfrak g$ is a rational Lie subalgebra and
$H:=\exp(\mathfrak h)$, then $H\cap\Gamma$ is a lattice in $H$ and
\[
\operatorname{span}_{\mathbb Q}\bigl(\log(H\cap\Gamma)\bigr)
=
\mathfrak h\cap\mathfrak g_{\mathbb Q}.
\]
Conversely, a connected Lie subgroup $H\leq G$ is rational if and only if
$H\cap\Gamma$ is a lattice in $H$, equivalently, if and only if
$H\Gamma/\Gamma$ is closed.
\item If $\mathfrak h$ is a rational ideal, $H:=\exp(\mathfrak h)$, and
$\pi:G\to G/H$ is the quotient homomorphism, then $G/H$ is connected and
simply connected, $\pi(\Gamma)$ is a lattice in $G/H$, and
the rational Lie algebra of $G/H$
is $(\mathfrak g_{\mathbb Q}+\mathfrak h)/\mathfrak h$.
\item The commutator subgroup $[G,G]$ is rational, and
$\Gamma[G,G]/[G,G]$ is a lattice in the vector group $G/[G,G]$.
\end{enumerate}
\end{proposition}

We next record two standard consequences of the structure theory of
connected and simply connected nilpotent Lie groups. The first concerns
connected Lie subgroups; see
\cite{CorwinGreenleaf1990}.

\begin{lemma}\label{lem:connected-subgroups-simply-connected}
Let $G$ be a connected and simply connected nilpotent Lie group. If
$H\leq G$ is a connected Lie subgroup with Lie algebra $\mathfrak h$, then
$
H=\exp(\mathfrak h).
$
In particular, $H$ is closed, embedded, and simply connected.
\end{lemma}

The second records the lattice and quotient properties that will be used
below; see \cite[Chapter~II]{Raghunathan1972}.

\begin{lemma}\label{lem:malcev}
Let $G$ be a connected and simply connected nilpotent Lie group with lattice
$\Gamma$.
\begin{enumerate}[label=\textup{(\roman*)}]
\item If $H\leq G$ is a connected Lie subgroup, then
$H\Gamma/\Gamma$ is closed if and only if $H\cap\Gamma$ is a lattice in
$H$. In that case, if $\pi_\Gamma:G\to G/\Gamma$ denotes the quotient map,
then $H$ is the connected component containing $e$ in
$
\pi_\Gamma^{-1}(H\Gamma/\Gamma)=H\Gamma.
$

\item Let $M\leq G$ be connected and rational, and let
$H\trianglelefteq M$ be connected and rational. Put
$
\Gamma_M:=M\cap\Gamma,
$
and let $\pi:M\to M/H$ be the quotient map. Then
$\pi(\Gamma_M)$ is a lattice in $M/H$ and
$
H\Gamma_M
=
\pi^{-1}\bigl(\pi(\Gamma_M)\bigr).
$
\end{enumerate}
\end{lemma}

The following lemma identifies the compact torus used in the central induction.

\begin{lemma}\label{lem:rational-torus}
Let $G$ be a connected and simply connected nilpotent Lie group with lattice
$\Gamma$.
\begin{enumerate}[label=\textup{(\roman*)}]
\item If $H,Z\leq G$ are
connected rational subgroups
and $Z$ is central, then
$HZ$ is a connected rational subgroup, $H\trianglelefteq HZ$, and
\[
\operatorname{Lie}(HZ)
=
\operatorname{Lie}(H)+\operatorname{Lie}(Z).
\]
\item Let $M\leq G$ be connected and rational, and let
$H\trianglelefteq M$ be connected and rational. Assume that
$[M,M]\subseteq H$. Put
\[
\mathfrak m:=\operatorname{Lie}(M),
\qquad
\mathfrak h:=\operatorname{Lie}(H),
\qquad
\Gamma_M:=M\cap\Gamma,
\]
and let $\pi:M\to M/H$ be the quotient map. Then
\[
\Omega
:=
\log_{M/H}\bigl(\pi(\Gamma_M)\bigr)
\subseteq
\mathfrak m/\mathfrak h
\]
is a full lattice,
$H\Gamma_M$ is a closed normal subgroup of $M$, and
\begin{equation}\label{eq:torus-quotient-general}
M/(H\Gamma_M)
\cong
(\mathfrak m/\mathfrak h)/\Omega
\end{equation}
is a compact connected torus. Under this identification, the differentials
of
the characters of this torus
are exactly the elements of $\Omega^*$.
\end{enumerate}
\end{lemma}

\begin{proof}
For (i), the multiplication map $H\times Z\to G$, $(h,z)\mapsto hz$, is
continuous and has image $HZ$. Since $H\times Z$ is connected, $HZ$ is
connected. Since $Z$ is central, $HZ$ is a subgroup and
$H$ is normal in $HZ$. The Lie algebra of $HZ$ is the sum of the rational
subalgebras $\operatorname{Lie}(H)$ and $\operatorname{Lie}(Z)$.
Rationality then follows from the Mal'cev criterion.

For (ii), $M$ and $H$ are simply connected by
Lemma~\ref{lem:connected-subgroups-simply-connected}. By
Proposition~\ref{prop:malcev-package}, $\Gamma_M$ is a lattice in $M$ and
$\pi(\Gamma_M)$ is a lattice in $M/H$. Since $[M,M]\subseteq H$, the
quotient $M/H$ is
a simply connected abelian Lie group and is therefore
identified, via the exponential map, with the vector group
$\mathfrak m/\mathfrak h$.
Thus $\Omega$ is a full lattice.
By Lemma~\ref{lem:malcev}(ii),
$H\Gamma_M=\pi^{-1}(\pi(\Gamma_M))$.
The subgroup $\pi(\Gamma_M)$ is closed because it is a lattice and normal
because $M/H$ is abelian. Hence its inverse image $H\Gamma_M$ is closed
and normal. The quotient isomorphism theorem gives
\[
M/(H\Gamma_M)\cong(M/H)/\pi(\Gamma_M).
\]
Under the exponential identification $M/H\cong\mathfrak m/\mathfrak h$,
the subgroup $\pi(\Gamma_M)$ corresponds to $\Omega$, proving
\eqref{eq:torus-quotient-general}. The description of the character
differentials then follows from the standard duality between a vector torus
and the dual lattice of its defining lattice.
\end{proof}

\begin{lemma}\label{lem:character-lift}
Assume the hypotheses of Lemma~\ref{lem:rational-torus}(ii), and put
$K:=M/(H\Gamma_M)$. For every continuous character
$\chi:K\to\mathbb R/\mathbb Z$,
the pullback of $\chi$ along the quotient map
$M\to K$
has a unique continuous lift
$\widetilde f_\chi:M\to\mathbb R$ satisfying
$\widetilde f_\chi(e)=0$. This lift is a group homomorphism. Its differential
$\ell_\chi:=d_e\widetilde f_\chi$ annihilates $\mathfrak h$ and descends to
an element of $\Omega^*$. As $\chi$ ranges over the character group
$\widehat K$, the descended functionals are exactly the elements of
$\Omega^*$ and therefore span $(\mathfrak m/\mathfrak h)^*$ over
$\mathbb R$.
\end{lemma}

\begin{proof}
Since $M$ is simply connected, the pullback of $\chi$ lifts uniquely through
the covering map
$\mathbb R\to\mathbb R/\mathbb Z$
after fixing the value at $e$. For
$x,y\in M$, the defect
\[
\widetilde f_\chi(xy)
-
\widetilde f_\chi(x)
-
\widetilde f_\chi(y)
\]
is integer-valued and continuous on the connected space $M\times M$. It
vanishes at $(e,e)$ and is therefore identically zero. Hence
$\widetilde f_\chi$ is a continuous group homomorphism and is therefore
smooth. Since the pullback of $\chi$ vanishes modulo $\mathbb Z$ on
$H\Gamma_M$, the lift takes integer values on $H\Gamma_M$. Its restriction
to the connected group $H$ is therefore constant and is zero because
$\widetilde f_\chi(e)=0$. Thus
$\ell_\chi:=d_e\widetilde f_\chi$ annihilates $\mathfrak h$ and induces a
functional on $\mathfrak m/\mathfrak h$. Moreover,
$\widetilde f_\chi(\Gamma_M)\subseteq\mathbb Z$, so the induced functional
takes integer values on $\Omega$ and hence belongs to $\Omega^*$.
Conversely, every element of $\Omega^*$ defines a character of
$(\mathfrak m/\mathfrak h)/\Omega$. The final spanning assertion follows
because $\Omega^*$ is a full lattice in $(\mathfrak m/\mathfrak h)^*$.
\end{proof}

\begin{corollary}\label{cor:identity-orbit-comparison}
Let $G$ be a connected and simply connected nilpotent Lie group with lattice
$\Gamma$, and let $H,K\leq G$ be connected Lie subgroups. If both
$H\Gamma/\Gamma$ and $K\Gamma/\Gamma$ are closed and
$H\Gamma/\Gamma=K\Gamma/\Gamma$, then $H=K$.
\end{corollary}

\begin{proof}
The inverse images of the two closed orbits under the
quotient map $G\to G/\Gamma$ coincide. By
Lemma~\ref{lem:malcev}(i), their identity components are $H$ and $K$.
\end{proof}

\subsubsection{Horizontal tori and horizontal characters}
Let $G$ be a connected and simply connected nilpotent Lie group with lattice
$\Gamma$. Denote the abelianization map by
$q_{\mathrm{ab}}:G\to G_{\mathrm{ab}}:=G/[G,G]$, and put
$\Gamma_{\mathrm{ab}}:=q_{\mathrm{ab}}(\Gamma)$. By
Proposition~\ref{prop:malcev-package}(iv), $\Gamma_{\mathrm{ab}}$ is a
lattice in the vector group $G_{\mathrm{ab}}$. Let
$\mathfrak g_{\mathrm{ab}}:=\mathfrak g/[\mathfrak g,\mathfrak g]$ and
define the horizontal lattice by
\begin{equation}\label{eq:horizontal-lattice}
\Lambda_{\mathrm{hor}}
:=
\log_{G_{\mathrm{ab}}}(\Gamma_{\mathrm{ab}})
\subseteq
\mathfrak g_{\mathrm{ab}}.
\end{equation}
The \emph{horizontal torus} of $G/\Gamma$ is
\[
G/[G,G]\Gamma
\cong
G_{\mathrm{ab}}/\Gamma_{\mathrm{ab}}
\cong
\mathfrak g_{\mathrm{ab}}/\Lambda_{\mathrm{hor}}.
\]

A functional $\eta\in\mathfrak g^*$ is \emph{horizontal} if
$\eta([\mathfrak g,\mathfrak g])=0$.
Such a functional induces a unique functional on
$\mathfrak g_{\mathrm{ab}}$, which we also denote by $\eta$.
A horizontal functional is \emph{rational} if
$\eta(\mathfrak g_{\mathbb Q})\subseteq\mathbb Q$; equivalently, the
induced functional takes rational values on $\Lambda_{\mathrm{hor}}$.
It is
\emph{integral} if $\eta(\Lambda_{\mathrm{hor}})\subseteq\mathbb Z$.
Every rational horizontal functional has a positive
integer multiple that is integral.

If $\eta$ is integral and horizontal, then
\[
\rho_\eta(g\Gamma)
:=
\eta\bigl(\log_{G_{\mathrm{ab}}}(q_{\mathrm{ab}}(g))\bigr)
+
\mathbb Z
\]
defines a continuous map from $G/\Gamma$ to
$\mathbb R/\mathbb Z$. This map is the pullback of a continuous character
of the horizontal torus.
We call $\rho_\eta$ the horizontal
character associated with $\eta$. Equivalently,
$\rho_\eta(\exp X\Gamma)=\eta(X)+\mathbb Z$.
If $\eta\neq0$, this character, and hence $\rho_\eta$, is surjective.

\subsection{Polynomial mappings and equidistribution}

\subsubsection{Integer-valued polynomials}

A polynomial
$q\in\mathbb Q[t_1,\ldots,t_\ell]$
is \emph{integer-valued} if
$q(\mathbb Z^\ell)\subseteq\mathbb Z$.

If
$\mathbf r:\mathbb Z^s\longrightarrow\mathbb Z^\ell
$
has integer-valued polynomial coordinates, then
$q\circ\mathbf r$ is again an integer-valued polynomial. We use this fact
when reparameterizing polynomial mappings.

\subsubsection{Polynomial mappings into nilpotent groups}
We use the following ordered product convention for the polynomial mappings needed below.
Let $G$ be a nilpotent group. A map
$h:\mathbb Z^\ell\to G
$
is a \emph{polynomial mapping} if there exist fixed elements
$a_1,\ldots,a_r\in G$ and integer-valued polynomials
$q_1,\ldots,q_r$ such that
\begin{equation}\label{eq:Leibman-polynomial-definition}
h(\mathbf n)
=
a_1^{q_1(\mathbf n)}
\cdots
a_r^{q_r(\mathbf n)}.
\end{equation}

Such representations need not be unique, and factor order matters when
$G$ is nonabelian; constant exponent polynomials are allowed.

Polynomial mappings are closed under products, inverses, postcomposition
with homomorphisms, and polynomial reparameterizations; see
\cite{Leibman1998,Leibman2002}.

\subsubsection{Ordinary polynomial maps}

Let $V$ be a finite-dimensional real vector space. A map
$P:\mathbb R^m\to V$ is an \emph{ordinary polynomial map} if its coordinate
functions in one, and hence every, basis of $V$ are real polynomials.

\begin{lemma}\label{lem:polynomial-identity}
Let $V$ be a finite-dimensional real vector space, let $W\leq V$ be a
linear subspace, and let
$P:\mathbb R^m\to V$
be an ordinary polynomial map. If
$P(\mathbf n)\in W$ for every $\mathbf n\in\mathbb Z^m$,
then
$P(\mathbf t)\in W$
for every $\mathbf t\in\mathbb R^m$.
In particular, an ordinary polynomial map that vanishes on
$\mathbb Z^m$ vanishes identically.
\end{lemma}

\begin{proof}
Let $\pi:V\to V/W$ be the quotient map. For every
$\ell\in(V/W)^*$, the scalar polynomial
$\ell\circ\pi\circ P$ vanishes on $\mathbb Z^m$ and hence vanishes
identically, by induction on $m$. Since $(V/W)^*$ separates points,
$\pi\circ P=0$.
\end{proof}

\begin{lemma}\label{lem:polynomial-antiderivative}
Let $Q:\mathbb R^m\to\mathbb R$ be continuously differentiable. If every
partial derivative $\partial_jQ$ is an ordinary real polynomial, then
$Q$ is an ordinary real polynomial. More precisely,
\[
Q(\mathbf t)
=
Q(\mathbf 0)
+
\int_0^1
\sum_{j=1}^m
t_j\,\partial_jQ(s\mathbf t)\,ds.
\]
\end{lemma}

\begin{proof}
Applying the fundamental theorem of calculus to
$s\mapsto Q(s\mathbf t)$ gives the displayed identity. Its right-hand
side is polynomial in $\mathbf t$ because each $\partial_jQ$ is
polynomial.
\end{proof}

\subsubsection{Uniform and well distribution}

Let $X$ be a compact metric space and let $m$ be a Borel probability
measure on $X$. A sequence $(x_n)_{n\in\mathbb N}$ is
\emph{uniformly distributed with respect to $m$} if
\[
\lim_{N\to\infty}
\frac1N\sum_{n=1}^N f(x_n)
=
\int_X f\,dm
\]
for every $f\in C(X)$.

Let $L\leq\mathbb Z^r$ be a subgroup. A \emph{F{\o}lner sequence} in
$L$ is a sequence $(\Phi_N)_{N\geq1}$ of nonempty finite subsets of $L$
such that
\[
\lim_{N\to\infty}
\frac{|(\Phi_N+\mathbf u)\mathbin{\triangle}\Phi_N|}
     {|\Phi_N|}
=
0
\qquad
\text{for every }\mathbf u\in L.
\]
A mapping $x:L\to X$ is \emph{well distributed with respect to $m$} if,
for every F{\o}lner sequence $(\Phi_N)$ in $L$ and every $f\in C(X)$,
\[
\lim_{N\to\infty}
\frac1{|\Phi_N|}
\sum_{\mathbf n\in\Phi_N}f(x(\mathbf n))
=
\int_X f\,dm.
\]

For a compact abelian group $K$ with Haar probability measure $m_K$, the Weyl character criterion states that
$(x_n)_{n\in\mathbb N}$ is
uniformly distributed with respect to $m_K$
if and only if
\[
\lim_{N\to\infty}
\frac1N\sum_{n=1}^N
\exp\bigl(2\pi i\chi(x_n)\bigr)
=
0
\]
for every nontrivial continuous character
$\chi:K\to\mathbb R/\mathbb Z$.

For a vector torus $V/\Omega$, its character group is identified with the
dual lattice $\Omega^*$ defined above; the character corresponding to
$\ell\in\Omega^*$ is
\[
\chi_\ell(v+\Omega)=\ell(v)+\mathbb Z.
\]

We use the following polynomial form of Weyl's criterion.

\begin{proposition}\label{prop:weyl-polynomial-criterion}
Let
$Q(t)=\beta_0+\beta_1t+\cdots+\beta_dt^d\in\mathbb R[t].
$
Then the sequence $(Q(n)+\mathbb Z)_{n\in\mathbb N}$ is uniformly
distributed in $\mathbb R/\mathbb Z$ if and only if at least one of the
coefficients $\beta_1,\ldots,\beta_d$ is irrational.
More generally, let $V$ be a finite-dimensional real vector space, let
$\Omega\leq V$ be a full lattice, and let
$P(t)=v_0+v_1t+\cdots+v_dt^d
$
be a polynomial with coefficients in $V$. Then
$(P(n)+\Omega)_{n\in\mathbb N}$ is uniformly distributed in $V/\Omega$
if and only if, for every nonzero $\ell\in\Omega^*$, at least one of the
numbers $\ell(v_1),\ldots,\ell(v_d)$ is irrational.
\end{proposition}

\begin{proof}
The implication for a polynomial with an irrational nonconstant
coefficient follows from Weyl's equidistribution theorem
\cite{Weyl1916}. If all nonconstant coefficients are rational, clearing
denominators yields a nontrivial character whose exponential average is a
nonzero constant. Applying the scalar criterion to $\ell\circ P$ for
every nonzero $\ell\in\Omega^*$ proves the vector statement.
\end{proof}

The irrationality condition in
Proposition~\ref{prop:weyl-polynomial-criterion} is preserved under every
affine reparameterization $n\mapsto an+b$, where
$a\in\mathbb Z\setminus\{0\}$ and $b\in\mathbb Z$. 
Indeed, for each nonzero $\ell\in\Omega^*$, let $j$ be the largest index
for which $\ell(v_j)$ is irrational. The coefficient of $n^j$ in
$\ell(P(an+b))$ is $a^j\ell(v_j)$ plus a rational number and is therefore
irrational.

\section{Affine reduction and exact interpolation}
\label{sec:affine-suspension}

We first replace the nilrotation in Theorem~\ref{thm:main} by a
unipotent affine system on a nilmanifold presented by a connected and
simply connected group. We then realize this affine system as a
translation on a larger nilmanifold using the suspension construction
\cite[Section~4]{DaniShahSharma2015}. The construction also gives the
interpolation and logarithmic derivative formulas used below.

\subsection{The affine model}
Let $N$ be a nilpotent Lie group with lattice $\Lambda$. If
$A\in\operatorname{Aut}(N)$ satisfies $A(\Lambda)=\Lambda$ and $b\in N$,
then
\[
F_{b,A}(x\Lambda):=bA(x)\Lambda
\]
defines an affine transformation of $N/\Lambda$. It is called
\emph{unipotent} when $A$ is unipotent.
Its inverse is induced by $x\mapsto A^{-1}(b^{-1}x)$.

The following lemma gives the affine model used below.

\begin{lemma}\label{lem:affine-reduction}
Under the assumptions of Theorem~\ref{thm:main}, there exist a connected and
simply connected nilpotent Lie group $N$, a lattice $\Lambda\leq N$, a
unipotent automorphism $A\in\operatorname{Aut}(N)$ satisfying
$A(\Lambda)=\Lambda$, and an element $b\in N$ such that $(X,T)$ is
topologically conjugate to the affine system
\begin{equation}\label{eq:affine-map}
F:N/\Lambda\longrightarrow N/\Lambda,
\qquad
F(x\Lambda)=bA(x)\Lambda.
\end{equation}
In particular, $F$ is minimal.
\end{lemma}

\begin{proof}
Write $X=G/\Gamma$, and suppose that $T$ is induced by left translation by
$a\in G$. Put $N_0:=G^0$. Since $X$ is connected,
it follows from
Lemma~\ref{lem:nilmanifold-components} that
$G=N_0\Gamma$.
Let $\Lambda_0:=N_0\cap\Gamma$. The natural map
\[
\phi_0:N_0/\Lambda_0\longrightarrow G/\Gamma,
\qquad
\phi_0(x\Lambda_0)=x\Gamma,
\]
is a homeomorphism, and $\Lambda_0$ is a lattice in $N_0$.

Choose $b_0\in N_0$ and $\gamma\in\Gamma$ such that
$a=b_0\gamma^{-1}$.
Define
\[
A_0:N_0\longrightarrow N_0,
\qquad
A_0(x):=\gamma^{-1}x\gamma.
\]
Equivalently,
$A_0=\operatorname{Int}(\gamma^{-1})|_{N_0}$. Since $N_0$ is normal in
$G$ and $\gamma\in\Gamma$, the automorphism $A_0$ preserves
$\Lambda_0=N_0\cap\Gamma$. Moreover, under $\phi_0$, left translation by
$a$ corresponds to
\[
x\Lambda_0\longmapsto b_0A_0(x)\Lambda_0.
\]
Indeed,
\[
ax\Gamma
=b_0\gamma^{-1}x\Gamma
=b_0\gamma^{-1}x\gamma\Gamma
=b_0A_0(x)\Gamma.
\]

We next verify that $A_0$ is unipotent.
Define the smooth map $\delta:G\to G$ by
\[
\delta(x):=\gamma^{-1}x\gamma x^{-1}=[\gamma,x^{-1}].
\]
Since $\delta(G_j)\subseteq G_{j+1}$, if $G$ has nilpotency step $s$,
then $\delta^{\circ s}\equiv e$. As $\delta(e)=e$,
\[
0=d_e(\delta^{\circ s})=(d_e\delta)^s.
\]
On $\operatorname{Lie}(N_0)$,
\[
d_e\delta=d_eA_0-\operatorname{id}_{\operatorname{Lie}(N_0)}.
\]
Hence $A_0$ is unipotent.

Let $\pi:N\to N_0$ be the universal covering homomorphism. By
Lemma~\ref{lem:universal-cover-package},
$\Lambda:=\pi^{-1}(\Lambda_0)$ is a lattice in $N$, and $\pi$ induces a
homeomorphism
\[
\overline\pi:N/\Lambda\longrightarrow N_0/\Lambda_0.
\]
The automorphism $A_0$ has a unique lift
$A\in\operatorname{Aut}(N)$ satisfying
\[
\pi\circ A=A_0\circ\pi.
\]
The lift $A$ is unipotent and satisfies $A(\Lambda)=\Lambda$. Choose
$b\in N$ with $\pi(b)=b_0$, and define $F$ by
\eqref{eq:affine-map}. Then
\[
\Phi:=\phi_0\circ\overline\pi:N/\Lambda\longrightarrow G/\Gamma
\]
is a homeomorphism satisfying $\Phi\circ F=T\circ\Phi$. Indeed, for every
$x\in N$,
\[
\begin{aligned}
\Phi(F(x\Lambda))
&=b_0A_0(\pi(x))\Gamma\\
&=b_0\gamma^{-1}\pi(x)\gamma\Gamma\\
&=a\pi(x)\Gamma\\
&=T(\Phi(x\Lambda)).
\end{aligned}
\]
Thus $(X,T)$ and $(N/\Lambda,F)$ are topologically conjugate. The minimality
of $F$ follows from that of $T$.
\end{proof}

\subsection{Nilpotent suspension and exact interpolation}
\label{sec:suspension}
Let $N,\Lambda,A,b$, and $F$ be as in
Lemma~\ref{lem:affine-reduction}. Put
$\mathfrak n:=\operatorname{Lie}(N)$ and $A_*:=d_eA$.
Since $A_*$ is unipotent, its logarithm
\[
\mathcal D:=\log A_*
=
\sum_{q\geq1}\frac{(-1)^{q+1}}{q}
(A_*-\operatorname{id}_{\mathfrak n})^q
\]
is a finite sum and is nilpotent.

To see that $\mathcal D$ is a derivation, fix $X,Y\in\mathfrak n$ and define
\[
P(r)
:=
\exp(r\mathcal D)[X,Y]
-
[\exp(r\mathcal D)X,\exp(r\mathcal D)Y]
\qquad (r\in\mathbb R).
\]
Then $P$ is a vector-valued polynomial. It vanishes for every
$r\in\mathbb Z$, since $\exp(r\mathcal D)=A_*^r$ is a Lie algebra
automorphism.
Hence $P$ vanishes identically. Differentiating at $r=0$ shows
\[
\mathcal D[X,Y]
=
[\mathcal D X,Y]+[X,\mathcal D Y].
\]

For $u\in\mathbb R$, set $A_*^u:=\exp(u\mathcal D)$. Using the exponential
diffeomorphism of $N$, define
\[
A^u(\exp X):=\exp(A_*^uX)
\qquad
(X\in\mathfrak n).
\]
Then $(A^u)_{u\in\mathbb R}$ is a one-parameter group of automorphisms of
$N$, and $A^1=A$.

Define
\[
\widehat N
:=
N\rtimes_{(A^u)_{u\in\mathbb R}}\mathbb R,
\qquad
\widehat\Lambda
:=
\Lambda\rtimes_{(A^r)_{r\in\mathbb Z}}\mathbb Z,
\]
with multiplication
\[
(x,u)(y,v)
=
\bigl(xA^u(y),u+v\bigr).
\]

\begin{lemma}
\label{lem:nilpotent-suspension-package}
The Lie group $\widehat N$ is connected, simply connected, and nilpotent,
and $\widehat\Lambda$ is a discrete cocompact subgroup of $\widehat N$.
The map
\[
\iota:N/\Lambda\longrightarrow\widehat N/\widehat\Lambda,
\qquad
\iota(x\Lambda):=(x,0)\widehat\Lambda,
\]
is a homeomorphism onto the fiber over $0+\mathbb Z$ of the natural map
\[
\widehat N/\widehat\Lambda\longrightarrow\mathbb R/\mathbb Z,
\qquad
(x,u)\widehat\Lambda\longmapsto u+\mathbb Z.
\]
Moreover, if $\alpha=(b,1)$, then the nilrotation $T_\alpha$ induced by $\alpha$ preserves this fiber and
\[
T_\alpha\circ\iota=\iota\circ F.
\]
\end{lemma}

\begin{proof}
Let
\[
\widehat{\mathfrak n}
:=
\mathfrak n\rtimes_{\mathcal D}\mathbb R S
\]
be the Lie algebra of $\widehat N$, so that
\[
[S,X]=\mathcal D X
\qquad
(X\in\mathfrak n).
\]
Choose $s,r\in\mathbb N$ such that
$\mathfrak n_{s+1}=\{0\}$ and $\mathcal D^r=0$. We claim that every Lie
monomial in $\widehat{\mathfrak n}$ of length greater than $sr$ vanishes.

A Lie monomial containing no letter from $\mathfrak n$ vanishes as soon as
its length is at least two. Suppose that a monomial contains
$q\geq1$ letters $X_1,\ldots,X_q\in\mathfrak n$ and $a$ letters $S$.
Repeatedly applying the derivation identity expresses such a monomial as a
linear combination of terms of the form
\[
[\mathcal D^{a_1}X_1,\ldots,\mathcal D^{a_q}X_q],
\qquad
a_1+\cdots+a_q=a,
\]
with possibly different bracketings. Such a term is zero if $q>s$. If
$q\leq s$ and $q+a>sr$, then
\[
a>q(r-1),
\]
so some $a_i\geq r$, and the term is again zero. Thus
$\widehat{\mathfrak n}$ is nilpotent, and hence $\widehat N$ is nilpotent.
The underlying manifold of $\widehat N$ is $N\times\mathbb R$, so it is
connected and simply connected. Since
$A^m(\Lambda)=\Lambda$ for every $m\in\mathbb Z$,
$\widehat\Lambda$ is a subgroup, and it is discrete.

Let
\[
\mathcal T_A
:=
\bigl((N/\Lambda)\times\mathbb R\bigr)/\!\sim,
\qquad
(x\Lambda,u+m)\sim(A^m(x)\Lambda,u)
\quad (m\in\mathbb Z).
\]
The map
\[
\Theta:\widehat N/\widehat\Lambda\longrightarrow\mathcal T_A,
\qquad
\Theta\bigl((x,u)\widehat\Lambda\bigr)
:=
[A^{-u}(x)\Lambda,u],
\]
is well defined. Indeed, right multiplication by
$(\lambda,m)\in\widehat\Lambda$ changes the representative to
\[
\bigl(xA^u(\lambda),u+m\bigr),
\]
whose image is equivalent to $[A^{-u}(x)\Lambda,u]$. Its inverse is
\[
[x\Lambda,u]\longmapsto(A^u(x),u)\widehat\Lambda.
\]
Changing $x$ to $x\lambda$ changes the image by right multiplication
by $(\lambda,0)$, while the defining relation of $\mathcal T_A$
corresponds to right multiplication by $(e,m)$.
Thus $\widehat N/\widehat\Lambda$ is homeomorphic to the mapping torus
\[
\bigl((N/\Lambda)\times[0,1]\bigr)/
\bigl((x\Lambda,1)\sim(A(x)\Lambda,0)\bigr)
\]
and is therefore compact.

Under $\Theta$, one has
\[
\Theta(\iota(x\Lambda))=[x\Lambda,0],
\]
so $\iota$ is a homeomorphism onto the fiber over $0+\mathbb Z$.
Moreover,
\[
\begin{aligned}
T_\alpha(\iota(x\Lambda))
&=\alpha(x,0)\widehat\Lambda\\
&=(bA(x),1)\widehat\Lambda\\
&=(bA(x),0)\widehat\Lambda\\
&=\iota(F(x\Lambda)),
\end{aligned}
\]
where the third equality follows by right multiplication by
$(e,-1)\in\widehat\Lambda$.
\end{proof}

Set
\[
\tau:=(e,1),
\qquad
\alpha:=(b,1)\in\widehat N,
\]
and write
\[
T_0:=\log\tau,
\qquad
X_0:=\log\alpha,
\qquad
v:=X_0-T_0.
\]
The differential at the identity of the time homomorphism
$\widehat N\to\mathbb R$, $(x,u)\mapsto u$, maps both $X_0$ and $T_0$ to
$1$; hence $v\in\mathfrak n$. Identifying $N$ with the subgroup
$N\times\{0\}$ of $\widehat N$, define
\begin{equation}\label{eq:def-beta}
\beta(u):=\alpha^u\tau^{-u}\in N
\qquad
(u\in\mathbb R).
\end{equation}

The following lemma gives the required interpolation and tangent
directions.

\begin{lemma}\label{lem:interpolation}
For every $r\in\mathbb Z$ and $x\in N$,
\begin{equation}\label{eq:F-interpolation}
F^r(x\Lambda)
=
\beta(r)A^r(x)\Lambda.
\end{equation}
Moreover, for every $u\in\mathbb R$,
\begin{equation}\label{eq:beta-log-derivative}
\beta(u)^{-1}\beta'(u)
=
A_*^uv.
\end{equation}
Define
\begin{equation}\label{eq:def-g}
g(\mathbf t)
:=
\bigl(\beta(p_1(\mathbf t)),\ldots,\beta(p_d(\mathbf t))\bigr)
\in N^d
\qquad
(\mathbf t\in\mathbb R^m),
\end{equation}
and define linear maps
$E_{\mathbf t},\Psi_{j,\mathbf t}:\mathfrak n\to\mathfrak n^d$ by
\begin{align}
E_{\mathbf t}(Y)
&:=
\bigl(A_*^{p_1(\mathbf t)}Y,\ldots,A_*^{p_d(\mathbf t)}Y\bigr),
\label{eq:def-E}\\
\Psi_{j,\mathbf t}(Y)
&:=
\bigl(
\partial_jp_1(\mathbf t)A_*^{p_1(\mathbf t)}Y,
\ldots,
\partial_jp_d(\mathbf t)A_*^{p_d(\mathbf t)}Y
\bigr).
\label{eq:def-Psi}
\end{align}
Then
\begin{equation}\label{eq:g-log-derivative}
g(\mathbf t)^{-1}\partial_jg(\mathbf t)
=
\Psi_{j,\mathbf t}(v)
\qquad
(1\leq j\leq m).
\end{equation}
For each fixed $Y\in\mathfrak n$ and each $1\leq j\leq m$, both
$\mathbf t\mapsto E_{\mathbf t}(Y)$ and
$\mathbf t\mapsto\Psi_{j,\mathbf t}(Y)$ are ordinary polynomial maps.

\end{lemma}

\begin{proof}

By Lemma~\ref{lem:nilpotent-suspension-package},
\[
\iota(F^r(x\Lambda))
=
T_\alpha^r(\iota(x\Lambda))
=
\alpha^r(x,0)\widehat\Lambda.
\]
Since $\beta(r)=\alpha^r\tau^{-r}$, one has
$\alpha^r=(\beta(r),r)$, and hence
\[
\alpha^r(x,0)\widehat\Lambda
=
\bigl(\beta(r)A^r(x),r\bigr)\widehat\Lambda
=
\bigl(\beta(r)A^r(x),0\bigr)\widehat\Lambda.
\]
The injectivity of $\iota$ now implies \eqref{eq:F-interpolation}.

By \eqref{eq:def-beta},
\[
\beta(u)=\exp(uX_0)\exp(-uT_0).
\]
Using the product formula for left logarithmic derivatives, we obtain
\[
\begin{aligned}
\beta(u)^{-1}\beta'(u)
&=
\operatorname{Ad}(\exp(uT_0))X_0-T_0\\
&=
\operatorname{Ad}(\tau^u)(T_0+v)-T_0\\
&=
A_*^uv.
\end{aligned}
\]
Here $\operatorname{Ad}(\tau^u)$ fixes $T_0$ and restricts to $A_*^u$ on
$\mathfrak n$. This proves \eqref{eq:beta-log-derivative}. Applying
\eqref{eq:beta-log-derivative} coordinatewise and using the chain rule with
respect to $t_j$ gives \eqref{eq:g-log-derivative}.

Finally, since $\mathcal D$ is nilpotent, there exists $J\geq0$ such that
\[
A_*^{p_i(\mathbf t)}
=
\exp\bigl(p_i(\mathbf t)\mathcal D\bigr)
=
\sum_{q=0}^{J}\frac{p_i(\mathbf t)^q\mathcal D^q}{q!}.
\]
Thus $\mathbf t\mapsto E_{\mathbf t}(Y)$ is an ordinary polynomial map. Multiplying its
$i$-th coordinate by $\partial_jp_i(\mathbf t)$ shows that
$\mathbf t\mapsto\Psi_{j,\mathbf t}(Y)$ is also an ordinary polynomial map.
\end{proof}

We shall use \eqref{eq:F-interpolation} to identify the original
polynomial slices and \eqref{eq:g-log-derivative} to control their
tangent directions.

\section{Orbit components and the horizontal obstruction}
\label{sec:components}

The following moving-subnilmanifold form of Leibman's polynomial orbit
theorem follows from \cite[Theorem~B$^*$]{Leibman2005Zd}; see also
\cite[p.~375]{BergelsonLeibmanLesigne2008}.

\begin{theorem}\label{thm:Leibman-orbit-closure}
Let $U/\Theta$ be a compact nilmanifold, let
$V\subseteq U/\Theta$ be a connected closed subnilmanifold, and let
$h:\mathbb Z^\ell\to U$ be a polynomial mapping in the sense of
\eqref{eq:Leibman-polynomial-definition}. Then there exist a connected
closed subgroup $H\leq U$, an integer $J\geq1$, and elements
$y_0,\ldots,y_{J-1}\in U$ such that
\[
\overline{\bigcup_{\mathbf n\in\mathbb Z^\ell}h(\mathbf n)V}
=
\bigsqcup_{\nu=0}^{J-1}Hy_\nu\Theta/\Theta.
\]
The orbits on the right are precisely the connected components of the
closure.
\end{theorem}

\begin{proposition}\label{prop:moving-subnilmanifold-haar}
Let $X=G/\Gamma$ be a compact nilmanifold, let $V\subseteq X$ be a
connected subnilmanifold, and let
$h:\mathbb Z^\ell\to G$ be a polynomial mapping. Put
\[
C:=\overline{\bigcup_{\mathbf n\in\mathbb Z^\ell}h(\mathbf n)V}.
\]
If $C$ is connected, then $C$ is a connected subnilmanifold and,
for every F{\o}lner sequence $(\Phi_N)$ in $\mathbb Z^\ell$ and every
$\varphi\in C(X)$,
\[
\lim_{N\to\infty}\frac1{|\Phi_N|}
\sum_{\mathbf n\in\Phi_N}
\int_V \varphi(h(\mathbf n)x)\,dm_V(x)
=
\int_C \varphi\,dm_C.
\]
\end{proposition}

\begin{proof}
This is the connected case of the moving-subnilmanifold
equidistribution theorem
\cite{BergelsonLeibmanLesigne2008}.
\end{proof}

View
$N/\Lambda$ as the fiber over $0+\mathbb Z$ of
$\widehat N/\widehat\Lambda$, and put
$V:=\Delta_{N/\Lambda}^d$. The subgroup $\Delta_N^d$ is connected and
$\Delta_N^d\cap\widehat\Lambda^d=\Delta_\Lambda^d$; hence, by
Lemma~\ref{lem:malcev}(i), $V$ is a connected closed
subnilmanifold of $(\widehat N/\widehat\Lambda)^d$.

Define
\[
h(\mathbf n)
:=
\bigl(
\alpha^{p_1(\mathbf n)},\ldots,
\alpha^{p_d(\mathbf n)}
\bigr).
\]
Each coordinate is polynomial, so
$h:\mathbb Z^m\to\widehat N^d$ is a polynomial mapping.

We next identify $h(\mathbf n)V$ with the corresponding polynomial
diagonal slice. Fix $\mathbf n\in\mathbb Z^m$ and $x\in N$. For
$1\leq i\leq d$,
\[
\alpha^{p_i(\mathbf n)}(x,0)(e,-p_i(\mathbf n))
=
\bigl(
\beta(p_i(\mathbf n))A^{p_i(\mathbf n)}(x),0
\bigr).
\]
Since $(e,-p_i(\mathbf n))\in\widehat\Lambda$, this element and
$\alpha^{p_i(\mathbf n)}(x,0)$ represent the same coset. By
\eqref{eq:F-interpolation}, its zero-fiber image is
$F^{p_i(\mathbf n)}(x\Lambda)$. Hence $h(\mathbf n)V$ is the
polynomial diagonal slice indexed by $\mathbf n$.

Let
\[
C
:=
\overline{\mathcal O_{\mathbf p}(\Delta_{N/\Lambda}^d)}
\subseteq
(N/\Lambda)^d.
\]
Applying Theorem~\ref{thm:Leibman-orbit-closure} to $h$ and $V$ in
$(\widehat N/\widehat\Lambda)^d$,
we obtain a connected closed subgroup
$H_0\leq\widehat N^d$, an integer $J\geq1$, and elements
$y_0,\ldots,y_{J-1}\in\widehat N^d$ such that
\[
C
=
\bigsqcup_{\nu=0}^{J-1}
H_0y_\nu\widehat\Lambda^d/\widehat\Lambda^d.
\]

Let
$\operatorname{time}:\widehat N^d\to\mathbb R^d$ be the product
homomorphism whose coordinates record the suspension times. Since
$\operatorname{time}(\widehat\Lambda^d)=\mathbb Z^d$, it induces a
continuous map
\[
\overline{\operatorname{time}}:
\widehat N^d/\widehat\Lambda^d
\longrightarrow
(\mathbb R/\mathbb Z)^d.
\]
Its zero fiber is naturally identified with $(N/\Lambda)^d$, since
$\widehat\Lambda^d\cap N^d=\Lambda^d$ and every class in this fiber has
a representative in $N^d$.

Every component above lies in the zero fiber. Fix $\nu$. Since
$y_\nu\widehat\Lambda^d$ lies in this fiber,
$\operatorname{time}(y_\nu)\in\mathbb Z^d$. For $g\in H_0$, the same is
true of $gy_\nu\widehat\Lambda^d$, so
$\operatorname{time}(g)\in\mathbb Z^d$. Thus
$\operatorname{time}(H_0)$ is a connected subset of $\mathbb Z^d$
containing $0$, and hence equals $\{0\}$. Therefore $H_0\leq N^d$.
After changing representatives, we may assume that $y_\nu\in N^d$.
Consequently,
\begin{equation}\label{eq:FU}
C
=
\bigsqcup_{\nu=0}^{J-1}
H_0y_\nu\Lambda^d/\Lambda^d.
\end{equation}

Since $p_i(\mathbf{0})=0$ for every $i$, the diagonal
$\Delta_{N/\Lambda}^d$ is contained in $C$. Since the diagonal is connected
and contains the identity coset $\Lambda^d$, it lies in the unique connected
component
of $C$ containing $\Lambda^d$. Denote this component by $C_0$ and reindex
the decomposition in \eqref{eq:FU} so that
$C_0=H_0y_0\Lambda^d/\Lambda^d$. Since $\Lambda^d\in C_0$, we have
$y_0\in H_0\Lambda^d$, and therefore
\begin{equation}\label{eq:C0}
C_0=H_0\Lambda^d/\Lambda^d.
\end{equation}
Since the orbit in \eqref{eq:C0} is closed, $H_0$ is rational relative
to $\Lambda^d$. 
By Lemma~\ref{lem:malcev}(i), $H_0$ is the connected component
containing the identity in the full inverse image
$H_0\Lambda^d$ of $C_0$ under the quotient map
$N^d\to (N/\Lambda)^d$.
Since $\Delta_{N/\Lambda}^d\subseteq C_0$, this full inverse image
contains $\Delta_N^d$. The subgroup $\Delta_N^d$ is connected and
contains the identity, so
$\Delta_N^d\subseteq H_0$.

\begin{lemma}\label{lem:finite-phases}
Let $Q\in\mathbb R[t_1,\ldots,t_m]$. If
$\bigl\{\exp\bigl(2\pi iQ(\mathbf n)\bigr):
\mathbf n\in\mathbb Z^m\bigr\}
$
is finite, then every nonconstant coefficient of $Q$ is rational. In
particular,
$\partial_jQ(\mathbf k)\in\mathbb Q$ for every
$\mathbf k\in\mathbb Z^m$ and every $1\leq j\leq m$.
\end{lemma}

\begin{proof}
Write $Q=\sum_{r=0}^D Q_r$, where $Q_r$ is homogeneous of degree $r$. 
For $\mathbf v\in\mathbb Z^m$, the polynomial
$q_{\mathbf v}(s):=Q(s\mathbf v)$ has finite image modulo one on
$\mathbb Z$.
By
Weyl's polynomial equidistribution theorem
(Proposition~\ref{prop:weyl-polynomial-criterion}), every nonconstant
coefficient of $q_{\mathbf v}$ is rational. The coefficient of $s^r$ is
$Q_r(\mathbf v)$, so
\[
Q_r(\mathbf v)\in\mathbb Q
\qquad
(r\geq1,\ \mathbf v\in\mathbb Z^m).
\]

Let $Q_r(\mathbf t)=\sum_{|\boldsymbol\alpha|=r}
 c_{\boldsymbol\alpha}\mathbf t^{\boldsymbol\alpha}$. For
$1\leq i\leq m$, let
$\Delta_iP(\mathbf t):=P(\mathbf t+\mathbf e_i)-P(\mathbf t)$. Since
$Q_r$ is homogeneous of degree $r$,
\[
c_{\boldsymbol\alpha}
=
\frac{1}{\alpha_1!\cdots\alpha_m!}
\bigl(
\Delta_1^{\alpha_1}\cdots
\Delta_m^{\alpha_m}Q_r
\bigr)(\mathbf{0}).
\]

Since the right-hand side is a rational linear combination of the
rational values $Q_r(\mathbf v)$ at integer points, every
$c_{\boldsymbol\alpha}$ is rational; the assertion about
$\partial_jQ(\mathbf k)$ follows.
\end{proof}

\begin{lemma}\label{lem:horizontal-obstruction}
Assume that the affine transformation $F(x\Lambda)=bA(x)\Lambda$ is minimal.
Let $\eta\in\mathfrak n^*$ be a nonzero functional that is rational
relative to the Mal'cev structure determined by $\Lambda$ and satisfies
\[
\eta([\mathfrak n,\mathfrak n])=0,
\qquad
\eta\circ A_*=\eta.
\]
Then $\eta(v)\notin\mathbb Q$, 
where $v=X_0-T_0$ is defined immediately before
\eqref{eq:def-beta}.
\end{lemma}

\begin{proof}
Since $\eta$ is rational and horizontal, there exists
$q\in\mathbb N$ such that $\eta_0:=q\eta$ is integral on the horizontal
lattice $\Lambda_{\mathrm{hor}}$. Let
\[
\rho_{\eta_0}:N/\Lambda\longrightarrow\mathbb R/\mathbb Z
\]
be the horizontal character constructed after
\eqref{eq:horizontal-lattice}. 
The map $\rho_{\eta_0}$ factors through a nontrivial character of the
connected horizontal torus and is therefore surjective. The corresponding lift to $N$,
\[
\widetilde\rho_{\eta_0}:N\longrightarrow\mathbb R,
\qquad
\widetilde\rho_{\eta_0}(\exp X):=\eta_0(X),
\]
is a group homomorphism and maps $\Lambda$ into $\mathbb Z$.

Since $\mathcal D=\log A_*$ is nilpotent, the map
$u\mapsto\eta_0\circ A_*^u$ is polynomial. It equals $\eta_0$ on
$\mathbb Z$, and hence
\[
\eta_0\circ A_*^u=\eta_0
\qquad (u\in\mathbb R).
\]

Using \eqref{eq:beta-log-derivative} and the fact that
$\widetilde\rho_{\eta_0}$ is a homomorphism, we obtain
\[
\frac{d}{du}
\widetilde\rho_{\eta_0}(\beta(u))
=
\eta_0\bigl(\beta(u)^{-1}\beta'(u)\bigr)
=
\eta_0(A_*^uv)
=
\eta_0(v).
\]
Since $\beta(0)=e$ and $\beta(1)=b$, integration over $[0,1]$ gives
$\widetilde\rho_{\eta_0}(b)=\eta_0(v)$.
Moreover, $\eta_0\circ A_*=\eta_0$ implies
$\widetilde\rho_{\eta_0}\circ A=\widetilde\rho_{\eta_0}$. Hence, for
every $x\in N$,
\[
\rho_{\eta_0}(F(x\Lambda))
=
\rho_{\eta_0}(x\Lambda)+\eta_0(v)+\mathbb Z.
\]
Thus $F$ factors onto the circle rotation by
$\eta_0(v)+\mathbb Z$. Minimality of $F$ forces this rotation to be
minimal, so $\eta_0(v)\notin\mathbb Q$ and hence
$\eta(v)\notin\mathbb Q$.
\end{proof}

\section{Central induction}
\label{sec:induction}

We retain the notation of Section~\ref{sec:components}: $C$ is the
polynomial diagonal orbit closure,
$C_0=H_0\Lambda^d/\Lambda^d$ is its component containing the identity
coset, and $g:\mathbb R^m\to N^d$ is the exact interpolation defined in
\eqref{eq:def-g}. Write
$\mathfrak h_0:=\operatorname{Lie}(H_0)$.

We prove Proposition~\ref{prop:strong-inclusion} by induction on the
nilpotency step, passing to the last central quotient and then
eliminating the residual torus by tangent and character arguments.

\begin{proposition}\label{prop:strong-inclusion}
For every $\mathbf t\in\mathbb R^m$, $g(\mathbf t)\in H_0$.
Consequently, $C=C_0$, so $C$ is connected.
\end{proposition}

\subsection{Passage to the last central quotient}

Let $N$ have nilpotency step $s\geq1$, and put
\[
Z:=N_s,
\qquad
\mathfrak z:=\operatorname{Lie}(Z).
\]
The subgroup $Z$ is connected, central, rational, and characteristic. In
particular, it is $A$-invariant.

\begin{lemma}\label{lem:central-quotient}
Let $q:N\to\overline N:=N/Z$ be the quotient homomorphism, put
$\overline\Lambda:=q(\Lambda)$, and let
$q_d:=q^{\times d}:N^d\to\overline N^d$. Suppose that, for the induced
affine system on $\overline N/\overline\Lambda$, the polynomial diagonal orbit closure
$\overline C$ is connected and the corresponding exact interpolation
$q_d\circ g$ takes values in the connected rational subgroup
$\overline H_0\leq\overline N^d$ that defines the component containing
the identity coset. Then
\[
q_d(H_0)=\overline H_0,\quad
g(\mathbb R^m)\subseteq M:=H_0Z^d.
\]
Moreover, $M$ is connected and rational, $H_0\trianglelefteq M$, and
\[
\operatorname{Lie}(M)=\mathfrak h_0+\mathfrak z^d.
\]
\end{lemma}

\begin{proof}
By Proposition~\ref{prop:malcev-package}(iii), $\overline\Lambda$ is a
lattice in the connected and simply connected nilpotent Lie group $\overline N$. 
Since $\mathfrak z$ is $\mathcal D$-invariant, each $A^u$ preserves
$Z$ and descends to $\overline N$. Consequently,
$(x,u)\mapsto(q(x),u)$ is a homomorphism of the corresponding
suspension groups and maps $\beta(u)$ to $q(\beta(u))$, the quotient
interpolation curve. Hence the exact interpolation in the quotient is
$q_d\circ g$.

Let $q_X:(N/\Lambda)^d\to(\overline N/\overline\Lambda)^d$ be the induced factor map.
Since $C$ is compact and $q_X$ maps each polynomial slice to the
corresponding quotient slice, one has $\overline C=q_X(C)$. Applying
$q_X$ to \eqref{eq:FU} expresses $\overline C$ as a finite union of
closed $q_d(H_0)$-orbits. After removing repetitions, these orbits form a finite disjoint
decomposition of $\overline C$ into relatively open sets. Since
$\overline C$ is connected and contains the identity coset,
\[
\overline C
=
q_d(H_0)\overline\Lambda^d/\overline\Lambda^d.
\]

The group $q_d(H_0)$ is connected and, by
Lemma~\ref{lem:connected-subgroups-simply-connected}, closed. Since the
orbit $q_d(H_0)\overline\Lambda^d/\overline\Lambda^d=\overline C$
is closed, $q_d(H_0)$ is rational relative to
$\overline\Lambda^d$.   As $\overline C$ is connected, both
$q_d(H_0)$ and $\overline H_0$ define the component containing the
identity, so
\[
q_d(H_0)\overline\Lambda^d/\overline\Lambda^d
=
\overline H_0\overline\Lambda^d/\overline\Lambda^d.
\]
Taking full inverse images under the quotient map
$\overline N^d\to\overline N^d/\overline\Lambda^d$ shows
\[
q_d(H_0)\overline\Lambda^d
=
\overline H_0\overline\Lambda^d.
\]
By Lemma~\ref{lem:malcev}(i), the connected components containing the
identity in the two sides are $q_d(H_0)$ and $\overline H_0$,
respectively. Hence
\[
q_d(H_0)=\overline H_0.
\]

By hypothesis, $q_d(g(\mathbf t))\in\overline H_0=q_d(H_0)$ for every
$\mathbf t\in\mathbb R^m$. Since $\ker q_d=Z^d$, it follows that
$g(\mathbf t)\in H_0Z^d=M$. Finally, $Z^d$ is connected, central, and
rational in $N^d$. By Lemma~\ref{lem:rational-torus}(i), $M$ is connected and rational,
$H_0\trianglelefteq M$, and
$\operatorname{Lie}(M)=\mathfrak h_0+\mathfrak z^d$.
\end{proof}

\subsection{Tangent directions forced by the polynomial slices}

\begin{lemma}\label{lem:tangent-invariance}
Under the hypotheses and conclusions of Lemma~\ref{lem:central-quotient},
\[
E_{\mathbf t}(\mathfrak n)\subseteq\mathfrak h_0
\qquad
(\mathbf t\in\mathbb R^m).
\]
If $A_\Delta:=A^{\times d}$, then
\[
A_\Delta(H_0)=H_0,
\qquad
(A_\Delta)_*(\mathfrak h_0)=\mathfrak h_0,
\qquad
(A_\Delta)_*(\mathfrak m)=\mathfrak m,
\]
where $\mathfrak m:=\operatorname{Lie}(M)$.
\end{lemma}

\begin{proof}
For $\mathbf n\in\mathbb Z^m$, let
\[
\Sigma_{\mathbf n}
:=
\left\{
\bigl(F^{p_1(\mathbf n)}x,\ldots,F^{p_d(\mathbf n)}x\bigr):
 x\in N/\Lambda
\right\}.
\]
This set is connected and contained in $C$, so it lies in one of the
connected components in \eqref{eq:FU}. At
$g(\mathbf n)\Lambda^d$, a lift to $N^d$ of the curve in
$\Sigma_{\mathbf n}$ induced by $Y\in\mathfrak n$ is
\[
s\longmapsto g(\mathbf n)\exp\bigl(sE_{\mathbf n}(Y)\bigr).
\]
By Lemma~\ref{lem:tangent-trivialization}(ii), the right logarithmic
derivative of this lifted curve at $s=0$ is
$\operatorname{Ad}(g(\mathbf n))E_{\mathbf n}(Y)$. The component
containing $g(\mathbf n)\Lambda^d$ is the orbit of that point under the
left action of $H_0$. Under the same local lift, right trivialization
identifies the tangent space at $g(\mathbf n)$ of the lifted $H_0$-orbit
with $\mathfrak h_0$. Therefore
\[
\operatorname{Ad}(g(\mathbf n))E_{\mathbf n}(\mathfrak n)
\subseteq
\mathfrak h_0.
\]
By Lemma~\ref{lem:central-quotient}, $g(\mathbf n)\in M$, and
$H_0\trianglelefteq M$. Hence $\operatorname{Ad}(g(\mathbf n))$ preserves
$\mathfrak h_0$, and we obtain
$E_{\mathbf n}(\mathfrak n)\subseteq\mathfrak h_0$ for every
$\mathbf n\in\mathbb Z^m$. For each fixed $Y\in\mathfrak n$, the map
$\mathbf t\mapsto E_{\mathbf t}(Y)$ is an ordinary polynomial map into
$\mathfrak n^d$. By Lemma~\ref{lem:polynomial-identity},
$E_{\mathbf t}(\mathfrak n)\subseteq\mathfrak h_0$ for all
$\mathbf t\in\mathbb R^m$.

The map $F^{\times d}$ preserves every $\Sigma_{\mathbf n}$ and hence
$C$. It sends the identity coset to
$b_\Delta\Lambda^d$, where
$b_\Delta:=(b,\ldots,b)\in\Delta_N^d\subseteq H_0$, and therefore
preserves $C_0$. Since
$F^{\times d}=\lambda_{b_\Delta}\circ A_\Delta$ and
$\lambda_{b_\Delta}$ preserves $C_0$, it follows that
$A_\Delta(C_0)=C_0$.

The inverse image of $C_0$ under the quotient map
$N^d\to N^d/\Lambda^d$ is $H_0\Lambda^d$. The automorphism
$A_\Delta$ preserves $\Lambda^d$ and fixes the identity, so it maps the
connected component containing the identity in this inverse image onto
itself. By Lemma~\ref{lem:malcev}(i), that component is $H_0$. Thus
$A_\Delta(H_0)=H_0$, and differentiation gives
$(A_\Delta)_*(\mathfrak h_0)=\mathfrak h_0$. Finally, $Z=N_s$ is
characteristic, so $A(Z)=Z$. Hence $A_\Delta(M)=M$ and
$(A_\Delta)_*(\mathfrak m)=\mathfrak m$.
\end{proof}

\begin{lemma}\label{lem:polynomial-tangent}
Under the hypotheses and conclusions of
Lemma~\ref{lem:central-quotient} and
Lemma~\ref{lem:tangent-invariance},
\[
\Psi_{j,\mathbf t}(\mathfrak n)\subseteq\mathfrak m
\qquad
(\mathbf t\in\mathbb R^m,\ 1\leq j\leq m).
\]
\end{lemma}

\begin{proof}
Fix $\mathbf k\in\mathbb Z^m$ and $1\leq j\leq m$, and define
\[
S_{\mathbf k,j}
:=
\{Y\in\mathfrak n:\Psi_{j,\mathbf k}(Y)\in\mathfrak m\}.
\]

First, $S_{\mathbf k,j}$ is rational.
Since $p_i(\mathbf k)$ and $\partial_jp_i(\mathbf k)$ are integers,
$A_*$ preserves the rational structure, and $\mathfrak m$ is rational,
the map $\Psi_{j,\mathbf k}$ is rational. Hence
$S_{\mathbf k,j}=\Psi_{j,\mathbf k}^{-1}(\mathfrak m)$
is rational.

Second, $S_{\mathbf k,j}$ is an ideal of $\mathfrak n$. For
$X,Y\in\mathfrak n$, the coordinatewise bracket identity gives
\[
[\Psi_{j,\mathbf k}(X),E_{\mathbf k}(Y)]
=
\Psi_{j,\mathbf k}([X,Y]).
\]
If $X\in S_{\mathbf k,j}$, then $\Psi_{j,\mathbf k}(X)\in\mathfrak m$, while
$E_{\mathbf k}(Y)\in\mathfrak h_0$ by Lemma~\ref{lem:tangent-invariance}. Since
$\mathfrak h_0$ is an ideal in $\mathfrak m$, the bracket on the left lies
in $\mathfrak h_0\subseteq\mathfrak m$. Thus
$[X,Y]\in S_{\mathbf k,j}$.

Third, $S_{\mathbf k,j}$ is invariant under $A_*$. The relation
\[
\Psi_{j,\mathbf k}(A_*Y)
=
(A_\Delta)_*\Psi_{j,\mathbf k}(Y)
\]
and Lemma~\ref{lem:tangent-invariance} imply
$A_*(S_{\mathbf k,j})\subseteq S_{\mathbf k,j}$. 
Since $A_*$ is invertible and $S_{\mathbf k,j}$ is finite-dimensional,
this inclusion is an equality. Thus $S_{\mathbf k,j}$ is invariant under
$A_*-\operatorname{id}_{\mathfrak n}$, hence under
$\mathcal D=\log A_*$ and under
$A_*^u=\exp(u\mathcal D)$ for every $u\in\mathbb R$.

Finally, $v\in S_{\mathbf k,j}$. Indeed,
$g(\mathbb R^m)\subseteq M$ by Lemma~\ref{lem:central-quotient}; hence
the partial left logarithmic derivatives of $g$ take values in
$\mathfrak m$. Equation~\eqref{eq:g-log-derivative} then yields
$\Psi_{j,\mathbf k}(v)\in\mathfrak m$.

Suppose that $S_{\mathbf k,j}\neq\mathfrak n$, and set
$H_{\mathbf k,j}:=\exp(S_{\mathbf k,j})$. 
By the preceding properties, $H_{\mathbf k,j}$ is a proper connected
closed normal rational subgroup of $N$ invariant under every $A^u$.
Let
\[
\pi_{\mathbf k,j}:N\to N/H_{\mathbf k,j}
\]
be the quotient homomorphism. By
Proposition~\ref{prop:malcev-package}(iii),
$N/H_{\mathbf k,j}$ is a nontrivial connected and simply connected
nilpotent Lie group, and $\pi_{\mathbf k,j}(\Lambda)$ is a lattice.
Hence
\[
(N/H_{\mathbf k,j})/\pi_{\mathbf k,j}(\Lambda)
\]
is a nontrivial compact nilmanifold and a factor of $N/\Lambda$.

Since $v\in S_{\mathbf k,j}$ and $S_{\mathbf k,j}$ is
$A_*^u$-invariant, \eqref{eq:beta-log-derivative} implies that the quotient
curve
$u\mapsto\pi_{\mathbf k,j}(\beta(u))$
has zero left logarithmic derivative. By
Lemma~\ref{lem:zero-logarithmic-derivative}, this curve is constant. 
Since $\beta(0)=e$ and $\beta(1)=b$, it follows that
$b\in H_{\mathbf k,j}$. Thus the transformation induced by $F$ on
this nontrivial quotient fixes the identity coset. But this quotient is
a nontrivial factor of the minimal system $(N/\Lambda,F)$ and is
therefore minimal, a contradiction. Hence $S_{\mathbf k,j}=\mathfrak n$.

We have proved
$\Psi_{j,\mathbf k}(\mathfrak n)\subseteq\mathfrak m$ for every
$\mathbf k\in\mathbb Z^m$. For each fixed $Y\in\mathfrak n$, the map
$\mathbf t\mapsto\Psi_{j,\mathbf t}(Y)$ is an ordinary polynomial map
into $\mathfrak n^d$. Applying Lemma~\ref{lem:polynomial-identity} to its
composition with the quotient map
$\mathfrak n^d\to\mathfrak n^d/\mathfrak m$ completes the proof.
\end{proof}

\subsection{The residual torus and its characters}

Let $\pi_0:M\to M/H_0$ be the quotient homomorphism, and put
\[
\Lambda_M:=M\cap\Lambda^d,
\qquad
K:=M/(H_0\Lambda_M).
\]
Since $M$ is a connected Lie subgroup of the connected and simply connected
nilpotent group $N^d$, it is closed and simply connected by
Lemma~\ref{lem:connected-subgroups-simply-connected}. Moreover,
$M=H_0Z^d$ with $Z^d$ central, so $[M,M]\subseteq H_0$. By
Lemma~\ref{lem:rational-torus}(ii), $H_0\Lambda_M$ is a closed normal
subgroup of $M$ and
\[
K\cong(\mathfrak m/\mathfrak h_0)/\Omega,
\qquad
\Omega:=\log_{M/H_0}\bigl(\pi_0(\Lambda_M)\bigr),
\]
where $\Omega$ is a full lattice. In particular, $K$ is a compact connected
torus.

\begin{lemma}\label{lem:residual-torus-characters}
For every continuous character $\chi\in\widehat K$, let
$\widetilde f_\chi:M\to\mathbb R$ be the homomorphic lift of the pullback
of $\chi$ along the quotient map $M\to K$ described in
Lemma~\ref{lem:character-lift}, and put
$\ell_\chi:=d_e\widetilde f_\chi\in\mathfrak m^*$. Then
\[
\ell_\chi(\mathfrak h_0)=0,
\qquad
\ell_\chi([\mathfrak m,\mathfrak m])=0,
\qquad
\ell_\chi(\log\Lambda_M)\subseteq\mathbb Z.
\]
Define
\[
Q_\chi(\mathbf t):=\widetilde f_\chi(g(\mathbf t))
\]
and, for $1\leq j\leq m$, define
\[
\eta_{\chi,j,\mathbf t}
:=
\ell_\chi\circ\Psi_{j,\mathbf t}
\in\mathfrak n^*.
\]
Then $Q_\chi$ is a real polynomial with $Q_\chi(\mathbf 0)=0$, and
\[
\partial_jQ_\chi(\mathbf t)
=
\eta_{\chi,j,\mathbf t}(v).
\]
For every $\mathbf k\in\mathbb Z^m$ and every $1\leq j\leq m$, the
functional $\eta_{\chi,j,\mathbf k}$ is rational, horizontal, and
$A_*$-invariant.
\end{lemma}

\begin{proof}
By Lemmas~\ref{lem:rational-torus}(ii) and
\ref{lem:character-lift}, $\widetilde f_\chi$ is a homomorphism,
$\ell_\chi$ annihilates $\mathfrak h_0$, and
$\ell_\chi(\log\Lambda_M)\subseteq\mathbb Z$. Since
$[\mathfrak m,\mathfrak m]\subseteq\mathfrak h_0$, it also annihilates
$[\mathfrak m,\mathfrak m]$. Moreover, since $M$ is rational,
\[
\operatorname{span}_{\mathbb Q}\log\Lambda_M
=
\mathfrak m\cap(\mathfrak n_{\mathbb Q})^d,
\]
so $\ell_\chi$ is rational relative to the Mal'cev structure on
$\mathfrak m$.

By Lemma~\ref{lem:polynomial-tangent}, the image of
$\Psi_{j,\mathbf t}$ lies in $\mathfrak m$. Functoriality of the left
logarithmic derivative under the homomorphism $\widetilde f_\chi$,
together with \eqref{eq:g-log-derivative}, yields
\[
\partial_jQ_\chi(\mathbf t)
=
\ell_\chi\bigl(g(\mathbf t)^{-1}\partial_jg(\mathbf t)\bigr)
=
\ell_\chi(\Psi_{j,\mathbf t}(v))
=
\eta_{\chi,j,\mathbf t}(v).
\]
The right-hand side is a real polynomial in $\mathbf t$. Since this holds
for every $j$ and $Q_\chi(\mathbf 0)=0$, it follows from
Lemma~\ref{lem:polynomial-antiderivative} that $Q_\chi$ is a real polynomial.

Fix $\mathbf k\in\mathbb Z^m$ and $1\leq j\leq m$. The map
$\Psi_{j,\mathbf k}$ is rational and takes $\mathfrak n_{\mathbb Q}$ into
$\mathfrak m\cap(\mathfrak n_{\mathbb Q})^d$. Hence
$\eta_{\chi,j,\mathbf k}$ is rational. For $X,Y\in\mathfrak n$, the
coordinatewise bracket identity from the proof of
Lemma~\ref{lem:polynomial-tangent} yields
\[
\eta_{\chi,j,\mathbf k}([X,Y])
=
\ell_\chi([\Psi_{j,\mathbf k}(X),E_{\mathbf k}(Y)])
=0,
\]
because $\Psi_{j,\mathbf k}(X)\in\mathfrak m$,
$E_{\mathbf k}(Y)\in\mathfrak h_0$, and
$[\mathfrak m,\mathfrak h_0]\subseteq\mathfrak h_0$. Thus
$\eta_{\chi,j,\mathbf k}$ is horizontal.

Finally, for $Y\in\mathfrak n$, differentiating the definition of
$E_{\mathbf t}$ implies
\[
\partial_jE_{\mathbf t}(Y)
=
\Psi_{j,\mathbf t}(\mathcal D Y).
\]
By Lemma~\ref{lem:tangent-invariance},
$E_{\mathbf t}(Y)\in\mathfrak h_0$ for every $\mathbf t$, so
$\Psi_{j,\mathbf t}(\mathcal D Y)\in\mathfrak h_0$.
Evaluating at $\mathbf t=\mathbf k$ and applying $\ell_\chi$ gives
\[
\eta_{\chi,j,\mathbf k}(\mathcal D Y)=0
\qquad (Y\in\mathfrak n).
\]
Hence $\eta_{\chi,j,\mathbf k}\circ\mathcal D=0$. Since
$A_*=\exp(\mathcal D)$, it follows that
\[
\eta_{\chi,j,\mathbf k}\circ A_*
=
\eta_{\chi,j,\mathbf k}.
\]

This finishes the proof.
\end{proof}

\subsection{Finite component phases and character elimination}

\begin{lemma}\label{lem:character-derivative-elimination}
For every $\chi\in\widehat K$, every $1\leq j\leq m$, and every
$\mathbf t\in\mathbb R^m$, one has
$\eta_{\chi,j,\mathbf t}=0$.
Consequently, $\Psi_{j,\mathbf t}(v)\in\mathfrak h_0$ whenever
$\mathbf t\in\mathbb R^m$ and $1\leq j\leq m$, and
$g(\mathbb R^m)\subseteq H_0$.
\end{lemma}

\begin{proof}
Fix $\chi\in\widehat K$ and $1\leq j\leq m$.
We first show that the phases
\[
\left\{
\exp\bigl(2\pi iQ_\chi(\mathbf n)\bigr):
 \mathbf n\in\mathbb Z^m
\right\}
\]
form a finite set. By Lemma~\ref{lem:central-quotient}, every $g(\mathbf n)$ lies in
$M$. If $g(\mathbf n)\Lambda^d$ and
$g(\mathbf n')\Lambda^d$ lie in the same component of
\eqref{eq:FU}, then
\[
g(\mathbf n')=h\,g(\mathbf n)\lambda
\]
for some $h\in H_0$ and $\lambda\in\Lambda^d$. 
Since $h,g(\mathbf n),g(\mathbf n')\in M$, the displayed identity
forces
$\lambda\in M\cap\Lambda^d=\Lambda_M$. As
$\widetilde f_\chi$ vanishes on $H_0$ and is integer-valued on
$\Lambda_M$,
\[
Q_\chi(\mathbf n')-Q_\chi(\mathbf n)\in\mathbb Z.
\]
Thus the phase depends only on the component containing the
corresponding basepoint. Since $C$ has finitely many components, the
phase set is finite.

By Lemma~\ref{lem:finite-phases},
\[
\eta_{\chi,j,\mathbf k}(v)
=
\partial_jQ_\chi(\mathbf k)
\in\mathbb Q
\qquad
(\mathbf k\in\mathbb Z^m).
\]
By Lemma~\ref{lem:residual-torus-characters}, the functional
$\eta_{\chi,j,\mathbf k}$ is rational, horizontal, and $A_*$-invariant. If it were
nonzero, Lemma~\ref{lem:horizontal-obstruction} would imply
$\eta_{\chi,j,\mathbf k}(v)\notin\mathbb Q$, a contradiction. Therefore
$\eta_{\chi,j,\mathbf k}=0$ for every $\mathbf k\in\mathbb Z^m$.

For each fixed $\chi$ and $j$, the map
$\mathbf t\mapsto\eta_{\chi,j,\mathbf t}$ is an ordinary polynomial map from
$\mathbb R^m$ to $\mathfrak n^*$. Lemma~\ref{lem:polynomial-identity} yields
$\eta_{\chi,j,\mathbf t}=0$ for every $\mathbf t\in\mathbb R^m$.

Since $\chi$ and $j$ were arbitrary, the preceding conclusion holds
for every $\chi\in\widehat K$, every $1\leq j\leq m$, and every
$\mathbf t\in\mathbb R^m$.
Moreover, Lemma~\ref{lem:polynomial-tangent} shows
$\Psi_{j,\mathbf t}(v)\in\mathfrak m$.
By Lemma~\ref{lem:character-lift}, as $\chi$ ranges over
$\widehat K$, the differentials $\ell_\chi$ span the dual of
$\mathfrak m/\mathfrak h_0$. Since
\[
\eta_{\chi,j,\mathbf t}
=
\ell_\chi\circ\Psi_{j,\mathbf t}
=
0,
\]
we have
\[
\ell_\chi\bigl(\Psi_{j,\mathbf t}(v)\bigr)=0
\qquad
(\chi\in\widehat K).
\]
Thus the image of $\Psi_{j,\mathbf t}(v)$ in
$\mathfrak m/\mathfrak h_0$ is annihilated by its full dual, and hence
\[
\Psi_{j,\mathbf t}(v)\in\mathfrak h_0.
\]

With $\pi_0$ as above, fix $\mathbf t\in\mathbb R^m$ and consider the
path
\[
c_{\mathbf t}(s):=\pi_0(g(s\mathbf t))
\qquad (0\leq s\leq1).
\]
By the chain rule and \eqref{eq:g-log-derivative}, its left logarithmic
derivative is
\[
\omega_L(c_{\mathbf t})(s)
=
\sum_{j=1}^m
 t_j\,d_e\pi_0\bigl(\Psi_{j,s\mathbf t}(v)\bigr)
=0.
\]
By Lemma~\ref{lem:zero-logarithmic-derivative}, $c_{\mathbf t}$ is constant.
Since $g(\mathbf 0)=e$, one has
$c_{\mathbf t}(0)=H_0$, and therefore $g(\mathbf t)\in H_0$. This holds
for every $\mathbf t\in\mathbb R^m$.
\end{proof}

\subsection{Completion of the induction}

\begin{proof}[Proof of Proposition~\ref{prop:strong-inclusion}]
We argue by induction on the nilpotency step of $N$, proving simultaneously
that $g(\mathbb R^m)\subseteq H_0$ and $C=C_0$.

If $N$ has step $0$, then $N=\{e\}$ and both conclusions are immediate.
Now let $N$ have step $s\geq1$, and assume the conclusions for all smaller
steps. Put $Z=N_s$. The quotient $N/Z$ has step at most $s-1$, and the induced affine
system is minimal because it is a factor of $(N/\Lambda,F)$. Applying the induction hypothesis to this
quotient verifies the hypotheses of Lemma~\ref{lem:central-quotient}. When
$s=1$, the quotient is trivial, so the same argument starts from $M=N^d$;
thus the abelian case is included in the induction.

The preceding arguments show that $g(\mathbb R^m)\subseteq H_0$.

For every $\mathbf n\in\mathbb Z^m$, the slice $\Sigma_{\mathbf n}$ is
connected and contains the point $g(\mathbf n)\Lambda^d\in C_0$. Since $C_0$
is a connected component of $C$, it follows that
$\Sigma_{\mathbf n}\subseteq C_0$.

Since $C_0$ is closed and
$C=\overline{\bigcup_{\mathbf n\in\mathbb Z^m}\Sigma_{\mathbf n}}$,
we have $C\subseteq C_0$; the reverse inclusion is immediate, so
$C=C_0$.
\end{proof}

\begin{proof}[Proof of Theorem~\ref{thm:main}]
By Lemma~\ref{lem:affine-reduction}, a coordinatewise conjugacy maps the
original polynomial diagonal orbit closure onto the corresponding
closure in the minimal unipotent affine model. The latter is connected
by Proposition~\ref{prop:strong-inclusion}; hence so is the original
orbit closure.
\end{proof}

\section{Recurrence and combinatorial consequences}
\label{sec:applications}

This section derives the recurrence consequences of
Theorem~\ref{thm:main}, followed by symbolic and coloring applications.

\subsection{Syndetic recurrence in prescribed residue classes}

A set $Q\subseteq\mathbb Z$ is \emph{syndetic} if there exists
$L\in\mathbb N$ such that every interval of $L$ consecutive integers meets
$Q$. For a nonempty open set $U\subseteq X$ and
$\mathbf p=(p_1,\ldots,p_d)\in\mathbb Z[t]^d$, write
\[
R_{\mathbf p}(U)
:=
\Bigl\{
 n\in\mathbb Z:
 U\cap\bigcap_{i=1}^dT^{-p_i(n)}U\neq\emptyset
\Bigr\}.
\]

\begin{corollary}\label{cor:return-times}
Let $(X,T)$ be a topological dynamical system, let $k,d\in\mathbb N$, and
assume that $T^k$ is minimal. Let $U\subseteq X$ be nonempty and open, and
let $p_1,\ldots,p_d\in\mathbb Z[t]$ satisfy $p_i(0)=0$. Then, for every
$j\in\mathbb Z$, the set
\[
Q_{\mathbf p,k,j}(U)
:=
\Bigl\{
q\in\mathbb Z:
U\cap\bigcap_{i=1}^dT^{-p_i(kq+j)}U\neq\emptyset
\Bigr\}
\]
is syndetic. Consequently,
\begin{equation}\label{eq:syndetic-slice}
R_{\mathbf p}(U)\cap(k\mathbb Z+j)
=
kQ_{\mathbf p,k,j}(U)+j
\end{equation}
is syndetic in $\mathbb Z$.
\end{corollary}

\begin{proof}
Retain one representative of each repeated polynomial. This does not
change $Q_{\mathbf p,k,j}(U)$. By the standing convention, the reduced
family consists of essentially distinct nonconstant polynomials vanishing
at the origin.  For $0\leq j<k$,
Proposition~\ref{prop:minimal-total-minimal} shows that the minimal
nilrotations appearing in the nilsystem formulation of
\cite[Theorem~4.23]{GlasscockEtAl2026} are totally minimal.
Theorem~\ref{thm:main} therefore verifies the connectedness hypothesis in
that equivalence theorem, whose prescribed-residue conclusion is exactly
the syndeticity of $Q_{\mathbf p,k,j}(U)$. For an arbitrary
$j\in\mathbb Z$,
write $j=j_0+k\ell$ with
$0\leq j_0<k$; then
$Q_{\mathbf p,k,j}(U)=Q_{\mathbf p,k,j_0}(U)-\ell$.
Equation~\eqref{eq:syndetic-slice} follows from the definitions, and the
image of a syndetic subset of $\mathbb Z$ under $q\mapsto kq+j$ is
syndetic.
\end{proof}

\subsection{Congruence invariance of polynomial correlations}
\label{subsec:measure-recurrence}

\subsubsection{Uniform Ces\`aro averages and characteristic factors}

\begin{lemma}\label{lem:folner-uc}
Let $a:\mathbb Z\to\mathbb C$ be bounded and let $L\in\mathbb C$. The
following statements are equivalent:
\begin{enumerate}[label=\textup{(\roman*)}]
\item for every F{\o}lner sequence $(\Phi_N)_{N\in\mathbb{N}}$ in $\mathbb Z$,
\[
\lim_{N\to\infty}
\frac{1}{|\Phi_N|}
\sum_{n\in\Phi_N}a(n)
=L;
\]
\item
\begin{equation}\label{eq:moving-interval-averages}
\lim_{N\to\infty}
\sup_{M\in\mathbb Z}
\left|
\frac{1}{N}
\sum_{n=M}^{M+N-1}a(n)-L
\right|
=0.
\end{equation}
\end{enumerate}
If $a$ has a limit along every F{\o}lner sequence, all these limits
coincide.
\end{lemma}

\begin{proof}
If (ii) fails, a sequence of intervals witnessing the failure is a
F{\o}lner sequence, so (i) fails. Conversely, assume (ii), and fix
$\varepsilon>0$. Choose
$L_0$ so that the average over every interval of length at least $L_0$ is
within $\varepsilon$ of $L$. Let $\mathcal I_N$ be the collection of
maximal interval components of $\Phi_N$, and let $S_N$ be the union of
those components whose lengths are less than $L_0$. Since
\[
|\mathcal I_N|
=\frac12\lvert(\Phi_N+1)\mathbin{\triangle}\Phi_N\rvert,
\]
we have
\[
|S_N|
\leq L_0|\mathcal I_N|
=\frac{L_0}{2}\lvert(\Phi_N+1)\mathbin{\triangle}\Phi_N\rvert
=o(|\Phi_N|).
\]
On every remaining component the average of $a-L$ has absolute value at
most $\varepsilon$, whereas
\[
\frac{1}{|\Phi_N|}
\left|\sum_{n\in S_N}(a(n)-L)\right|
\leq
(\lVert a\rVert_\infty+|L|)\frac{|S_N|}{|\Phi_N|}
=o(1).
\]
Thus the average over $\Phi_N$ differs from $L$ by at most
$\varepsilon+o(1)$, proving (i). Finally,
interleaving two F{\o}lner
sequences shows that two putative limits must agree.
\end{proof}

When these equivalent conditions hold, we write
$\UClim_{n\in\mathbb Z}a(n)=L$.

A \emph{pro-nilsystem} is a measure-preserving inverse limit of nilsystems,
with Haar measure at each finite stage and transition maps
intertwining the nilrotations. We use the following consequence of
polynomial characteristic factor theory.

\begin{lemma}\label{lem:polynomial-characteristic-pronil}
Let $(X,\mathcal B,\mu,T)$ be an ergodic system. Let
$q_1,\ldots,q_s\in\mathbb Z[t]$ be nonconstant and essentially distinct,
meaning that $q_i-q_{i'}$ is nonconstant whenever $i\neq i'$. No assumption
is made on their constant terms. Then there exists a pro-nilsystem factor
$\pi:(X,\mathcal B,\mu,T)
\to
(Z,\mathcal Z,\nu,R)
$
depending only on the system and the polynomial family, with the following
property. For $f_0,\ldots,f_s\in L^\infty(\mu)$, define
$f_i^Z\in L^\infty(\nu)$ by
\[
f_i^Z\circ\pi
=
\mathbb E(f_i\mid\pi^{-1}\mathcal Z)
\qquad (0\leq i\leq s),
\]
and put
\[
\begin{aligned}
a_X(n)
&:=
\int_X f_0(x)
\prod_{i=1}^s f_i(T^{q_i(n)}x)
\,d\mu(x),\\
a_Z(n)
&:=
\int_Z f_0^Z(z)
\prod_{i=1}^s f_i^Z(R^{q_i(n)}z)
\,d\nu(z).
\end{aligned}
\]
Then, for every F{\o}lner sequence $(\Phi_N)_{N\in\mathbb{N}}$ in $\mathbb Z$,
\begin{equation}\label{eq:characteristic-pronilfactor}
\lim_{N\to\infty}
\frac{1}{|\Phi_N|}
\sum_{n\in\Phi_N}
\bigl(a_X(n)-a_Z(n)\bigr)
=0.
\end{equation}
A single factor may be chosen simultaneously for any prescribed finite
collection of polynomial families.
\end{lemma}

\begin{proof}
Subtract the constants $q_i(0)$ and absorb the resulting shifts into the
functions. By \cite[Theorem~3]{Leibman2005Averages}, there exists $r$,
depending only on the polynomial family, such that
$\mathcal Z_{r-1}$ is characteristic along every F{\o}lner sequence.
A telescoping argument therefore allows
$f_1,\ldots,f_s$ to be replaced by their conditional expectations onto
$\mathcal Z_{r-1}$. After these replacements, the remaining product is
$\mathcal Z_{r-1}$-measurable, since this factor is $T$-invariant.
Hence $f_0$ may also be replaced by its conditional expectation onto
$\mathcal Z_{r-1}$.

By \cite[Theorem~10.1]{HostKra2005Nil}, the factor
$\mathcal Z_{r-1}$ is an inverse limit of nilsystems. 
A pro-nilsystem realization of $\mathcal Z_{r-1}$ gives
\eqref{eq:characteristic-pronilfactor}. For finitely many polynomial
families, the nesting of the Host--Kra factors allows us to take the
largest of the corresponding characteristic orders.
\end{proof}

\subsubsection{Polynomial reparameterization and stability}
\label{subsec:polynomial-reparameterization}

We use the following polynomial reparameterization result
\cite[Proposition~4.13]{GlasscockEtAl2026}.

\begin{proposition}\label{prop:connected-orbit-reparameterization}
Let $X=G/\Gamma$ be a compact nilmanifold, let $V\subseteq X$ be a
connected subnilmanifold, and let $g:\mathbb Z\to G$ be a polynomial
mapping. Put
\[
C_g(V)
:=
\overline{\bigcup_{n\in\mathbb Z}g(n)V}.
\]
If $C_g(V)$ is connected, then, for every nonconstant
$r\in\mathbb Z[t]$,
\begin{equation}\label{eq:connected-orbit-reparameterization}
\overline{\bigcup_{n\in\mathbb Z}g(r(n))V}
=
C_g(V).
\end{equation}
\end{proposition}

\begin{corollary}\label{cor:reparameterization}
Let $X=G/\Gamma$ be a compact connected nilmanifold, let
$T(g\Gamma)=ag\Gamma$ be a minimal nilrotation, and let
$\mathbf p=(p_1,\ldots,p_d)\in\mathbb Z[t]^d$ satisfy $p_i(0)=0$. If
$r\in\mathbb Z[t]$ is nonconstant, then
\[
\overline{\mathcal O_{\mathbf p\circ r}(\Delta_X^d)}
=
\overline{\mathcal O_{\mathbf p}(\Delta_X^d)},
\qquad
\mathbf p\circ r
:=
(p_1\circ r,\ldots,p_d\circ r).
\]
In particular, restricting the parameter to the progression $kn+j$, with
$k\in\mathbb Z\setminus\{0\}$ and $j\in\mathbb Z$, does not change the
polynomial diagonal orbit closure.
\end{corollary}

\begin{proof}
Apply Theorem~\ref{thm:main} and
Proposition~\ref{prop:connected-orbit-reparameterization} to the polynomial
mapping
\[
n\longmapsto
(a^{p_1(n)},\ldots,a^{p_d(n)})
\]
and the connected subnilmanifold $\Delta_X^d$.
\end{proof}

\begin{corollary}\label{cor:arithmetic-reparametrized-haar}
Let $X=G/\Gamma$ be a compact nilmanifold, let $V\subseteq X$ be a
connected subnilmanifold, and let $g:\mathbb Z\to G$ be a polynomial
mapping. Assume that
\[
C:=\overline{\bigcup_{n\in\mathbb Z}g(n)V}
\]
is connected. Then, for every $a\in\mathbb Z\setminus\{0\}$ and
$b\in\mathbb Z$,
\[
\overline{\bigcup_{n\in\mathbb Z}g(an+b)V}=C.
\]
Moreover, for every F{\o}lner sequence $(\Phi_N)$ in $\mathbb Z$ and every
$\varphi\in C(X)$,
\begin{equation}\label{eq:arithmetic-reparametrized-haar}
\lim_{N\to\infty}
\frac{1}{|\Phi_N|}
\sum_{n\in\Phi_N}
\int_V \varphi(g(an+b)x)\,dm_V(x)
=
\int_C \varphi\,dm_C.
\end{equation}
\end{corollary}

\begin{proof}
Apply Proposition~\ref{prop:connected-orbit-reparameterization} to
$r(n)=an+b$, followed by
Proposition~\ref{prop:moving-subnilmanifold-haar}.
\end{proof}

\begin{lemma}\label{lem:correlation-L1-stability}
Let $(Y,\nu)$ be a probability space, let $S_0,\ldots,S_d$ be
measure-preserving transformations, and suppose that
$f_i,f_i'\in L^\infty(\nu)$ satisfy
$\|f_i\|_\infty,\|f_i'\|_\infty\leq B$. Then
\[
\left|
\int_Y\prod_{i=0}^df_i(S_i y)\,d\nu(y)
-
\int_Y\prod_{i=0}^df_i'(S_i y)\,d\nu(y)
\right|
\leq
B^d\sum_{i=0}^d\|f_i-f_i'\|_1.
\]
Consequently, the same bound applies uniformly to every term of a
polynomial correlation sequence and to every average over a moving
interval.
\end{lemma}

\begin{proof}
Expand the difference as a telescoping sum, changing one factor at a time.
\end{proof}

\subsubsection{Progression invariance on pro-nilsystems}

\begin{proposition}\label{prop:pronil-progression-average}
Let $(Y,m_Y,R)$ be a pro-nilsystem, and suppose that $R^k$ is ergodic for
some $k\in\mathbb N$. 
Let $p_1,\ldots,p_d\in\mathbb Z[t]$ satisfy $p_i(0)=0$, and let
$f_0,\ldots,f_d\in L^\infty(m_Y)$. Define
\[
c(n)
:=
\int_Y f_0(y)\prod_{i=1}^d
f_i(R^{p_i(n)}y)\,dm_Y(y).
\]
Then there exists $L\in\mathbb C$, independent of $j$, such that
\begin{equation}\label{eq:pronil-progression-average}
\begin{aligned}
&\lim_{N\to\infty}
\sup_{M\in\mathbb Z}
\left|
\frac{1}{N}
\sum_{n=M}^{M+N-1}c(n)-L
\right|
=0,\\
&\lim_{N\to\infty}
\sup_{M\in\mathbb Z}
\left|
\frac{1}{N}
\sum_{q=M}^{H+N-1}c(kq+j)-L
\right|
=0
\qquad
\text{for every fixed }j\in\mathbb Z.
\end{aligned}
\end{equation}
\end{proposition}

\begin{proof}
\emph{Nilsystem case.}
Assume first that $Y=G/\Gamma$ is a nilsystem and $R=T_a$. Since $R^k$ is
ergodic, both $R^k$ and $R$ are minimal. Let $\kappa$ be the number of
connected components, let $Y_0$ be the component containing the identity
coset, and write
\[
Y_s:=R^sY_0
\qquad
(s\in\mathbb Z/\kappa\mathbb Z).
\]
The action of $R^k$ on the components is $s\mapsto s+k$, so its
transitivity implies
\begin{equation}\label{eq:component-coprimality}
\gcd(k,\kappa)=1.
\end{equation}
Moreover, $S:=R^\kappa|_{Y_0}$ is minimal: a proper nonempty closed
$S$-invariant subset of $Y_0$ would generate a proper nonempty closed
$R$-invariant subset of $Y$.

Since $a^\kappa\in G^0\Gamma$, set
\[
H:=\overline{\langle G^0,a^\kappa\rangle}.
\]
Then $H\Gamma=G^0\Gamma$, $H\cap\Gamma$ is a lattice in $H$, and
\[
Y_0\cong H/(H\cap\Gamma),
\]
with $S$ induced by $a^\kappa$. Although $H$ need not be connected,
$Y_0$ is connected, which is the hypothesis required in
Theorem~\ref{thm:main}. Let $m_0$ denote Haar probability measure on
$Y_0$.

Set the auxiliary polynomial $p_0\equiv0$.
For $0\leq u<\kappa$ and $0\leq i\leq d$, put
\[
P_{i,u}(\ell)
:=
\frac{p_i(\kappa\ell+u)-p_i(u)}{\kappa}
\in\mathbb Z[\ell].
\]
The integrality follows from the binomial theorem, and $P_{i,u}(0)=0$.
For $0\leq s<\kappa$, define
\[
f_{i,s,u}(y):=f_i(R^{s+p_i(u)}y)
\qquad (y\in Y_0).
\]
Since the components have equal Haar measure,
\begin{equation}\label{eq:component-correlation-splitting}
c(\kappa\ell+u)
=
\frac{1}{\kappa}
\sum_{s=0}^{\kappa-1}
\int_{Y_0}
\prod_{i=0}^d
f_{i,s,u}(S^{P_{i,u}(\ell)}y)
\,dm_0(y).
\end{equation}

Let
\[
C_u
:=
\overline{
\bigcup_{\ell\in\mathbb Z}
\bigl(
S^{P_{0,u}(\ell)}\times\cdots\times
S^{P_{d,u}(\ell)}
\bigr)
\Delta_{Y_0}^{d+1}
}.
\]
The polynomial $P_{0,u}$ is identically zero. By the common shift
identity, adding $M\ell$ to every $P_{i,u}$ does not change $C_u$.
Choose $M\in\mathbb Z$ outside the finite exceptional set for which one
of the shifted polynomials vanishes identically. Then
Theorem~\ref{thm:main} and
Proposition~\ref{prop:moving-subnilmanifold-haar} show that $C_u$ is a
connected subnilmanifold. If the $f_i$ are continuous, set
\[
F_{s,u}(x_0,\ldots,x_d)
:=
\prod_{i=0}^df_{i,s,u}(x_i),
\qquad
L_u
:=
\frac{1}{\kappa}
\sum_{s=0}^{\kappa-1}
\int_{C_u}F_{s,u}\,dm_{C_u}.
\]
For continuous $f_i$, applying
Corollary~\ref{cor:arithmetic-reparametrized-haar} to each $F_{s,u}$
shows that
$\ell\mapsto c(\kappa(a\ell+b)+u)$ has uniform Ces\`aro limit $L_u$
for every $a\in\mathbb Z\setminus\{0\}$ and $b\in\mathbb Z$.
For bounded measurable $f_i$, approximate the finitely many functions
$f_{i,s,u}$ in $L^1(m_0)$ by uniformly bounded continuous functions.
Lemma~\ref{lem:correlation-L1-stability} transfers this conclusion and
shows that $L_u$ is independent of the approximations. Hence
\begin{equation}\label{eq:component-slice-limit}
\UClim_{\ell\in\mathbb Z}
c\bigl(\kappa(a\ell+b)+u\bigr)
=
L_u
\qquad
(a\in\mathbb Z\setminus\{0\},\ b\in\mathbb Z).
\end{equation}

Splitting moving intervals into residue classes modulo $\kappa$ yields
\begin{equation}\label{eq:full-component-average}
\UClim_{n\in\mathbb Z}c(n)
=
\frac{1}{\kappa}
\sum_{u=0}^{\kappa-1}L_u.
\end{equation}
Fix $j\in\mathbb Z$.
For $0\leq t<\kappa$, write
\[
kt+j=\kappa v_t+u_t
\qquad
(0\leq u_t<\kappa).
\]
Then
$k(\kappa\ell+t)+j=\kappa(k\ell+v_t)+u_t$.
By \eqref{eq:component-coprimality}, the map $t\mapsto u_t$ permutes
$\{0,\ldots,\kappa-1\}$. Combining this observation with
\eqref{eq:component-slice-limit} and again splitting moving intervals into
residue classes yields
\begin{equation}\label{eq:sliced-component-average}
\UClim_{q\in\mathbb Z}c(kq+j)
=
\frac{1}{\kappa}
\sum_{u=0}^{\kappa-1}L_u,
\end{equation}
independently of $j$. This proves the proposition for nilsystems.

\emph{Passage to the inverse limit.}
Now let $Y$ be a pro-nilsystem. Successive common refinements yield
increasing finite-stage factor $\sigma$-algebras $\mathcal F_m$ such that
\[
\max_{0\leq i\leq d}
\left\lVert
f_i-\mathbb E(f_i\mid\mathcal F_m)
\right\rVert_1
\leq2^{-m}.
\] 
Put
\[
f_i^{(m)}:=\mathbb E(f_i\mid\mathcal F_m),
\qquad
B:=\max\Big\{1,\max_{0\leq i\leq d}\lVert f_i\rVert_\infty\Big\},
\]
and
\[
\delta_m
:=
B^d\sum_{i=0}^d\lVert f_i^{(m)}-f_i\rVert_1.
\]
Then $\delta_m\to0$. Lemma~\ref{lem:correlation-L1-stability} shows that
every full or sliced finite-stage correlation differs from its original
counterpart by at most $\delta_m$, uniformly in the sequence index and in
$j\in\mathbb Z$.

The transformation induced by $R^k$ on each finite stage factor is
ergodic. Let $L_m$ be the common uniform Ces\`aro limit at stage $m$.
Comparing stages $m$ and $m'$ and then passing to their uniform Ces\`aro
limits gives
\[
\lvert L_m-L_{m'}\rvert\leq\delta_m+\delta_{m'}.
\]
Thus $(L_m)$ is Cauchy; let $L:=\lim_{m\to\infty}L_m$. For either the full
correlation sequence, or the sliced correlation sequence corresponding to
a fixed $j\in\mathbb Z$, write $A_{N,H}$ and $A_{N,H}^{(m)}$ for the
respective moving-interval averages at the original level and at stage
$m$. Then
\[
\sup_{H\in\mathbb Z}
\lvert A_{N,H}-A_{N,H}^{(m)}\rvert
\leq\delta_m
\]
for every $N$. The finite-stage result gives convergence to $L_m$, uniformly
in $H$, in the case of the full sequence and, separately, for each fixed $j$ in the
sliced case. It follows that
\[
\limsup_{N\to\infty}\sup_{H\in\mathbb Z}
\lvert A_{N,H}-L\rvert
\leq\delta_m+\lvert L_m-L\rvert.
\]
Letting $m\to\infty$ proves
\eqref{eq:pronil-progression-average} for the full sequence and,
separately, for each fixed sliced sequence. The limiting value $L$ is
independent of $j$.
\end{proof}

\subsubsection{Proof of the measure-theoretic results}

\begin{proof}[Proof of Theorem~\ref{thm:measure-common-limit}]
Since $T^k$ is ergodic, so is $T$. Removing repeated polynomials does not
change the correlation sequence, since $\mathbf1_A^2=\mathbf1_A$.
After relabeling, write the resulting essentially distinct family of
nonconstant polynomials as $p_1,\ldots,p_d$.

For $0\leq j<k$, set
\[
r_{i,j}(t):=p_i(kt+j).
\]
Each sliced family $(r_{i,j})_i$ is again nonconstant and essentially
distinct. Apply Lemma~\ref{lem:polynomial-characteristic-pronil}
simultaneously to the original family and the $k$ sliced families. This
yields a single pro-nilsystem factor
$
\pi:(X,\mathcal B,\mu,T)
\to
(Z,\mathcal Z,\nu,R)
$
that is characteristic for every family. Since $R^k$ is a factor of $T^k$,
it is ergodic.

Choose $h\in L^\infty(\nu)$ with $0\leq h\leq1$ almost everywhere and
satisfying
\[
h\circ\pi
=
\mathbb E(\mathbf1_A\mid\pi^{-1}\mathcal Z),
\]
and define
\[
b(n)
:=
\int_Z h(z)\prod_{i=1}^d h(R^{p_i(n)}z)\,d\nu(z).
\]
By Proposition~\ref{prop:pronil-progression-average}, $b$ and, for each
fixed $j$, the sequence $q\mapsto b(kq+j)$ have the same uniform
Ces\`aro limit $C$.
By \eqref{eq:characteristic-pronilfactor}, applied to the original and
sliced families, the averages of $c-b$ and, for each fixed $j$, of
$q\mapsto c(kq+j)-b(kq+j)$ converge to zero
along every F{\o}lner sequence.
By Lemma~\ref{lem:folner-uc}, these convergences are uniform over moving
intervals, and \eqref{eq:measure-moving-intervals} follows.
\end{proof}

\begin{proof}[Proof of Corollary~\ref{cor:measure-positive-threshold}]
By polynomial multiple recurrence
\cite{BergelsonLeibman1996},
\[
\liminf_{N\to\infty}
\frac1N\sum_{n=0}^{N-1}c(n)>0.
\]
Since these averages converge to $C$ by
Theorem~\ref{thm:measure-common-limit}, we have $C>0$.

Fix $0<\varepsilon<C$. If, for some $j$, the set
$\{q:c(kq+j)>\varepsilon\}$ were not syndetic, it would miss arbitrarily
long intervals. The average on any such interval would be at most
$\varepsilon$, contradicting
\eqref{eq:measure-moving-intervals}.
\end{proof}

\begin{remark}
Each syndetic superlevel set is infinite, so $q$ may be chosen so that
$kq+j\neq0$.
\end{remark}

Theorem~\ref{thm:measure-common-limit} and
Corollary~\ref{cor:measure-positive-threshold} imply the invertible form of
\cite[Conjecture~1.6]{GlasscockEtAl2026}.  
The assumption that $T^k$ is ergodic cannot be replaced by
ergodicity of $T$ alone: for $k\geq2$, addition by $1$ on
$\mathbb Z/k\mathbb Z$ is ergodic but its $k$-th power is the identity, and
for $A=\{0\}$ and $p(n)=n$ the correlation vanishes at every time $kq+1$.

\subsection{Symbolic and coloring consequences}

Let $\mathcal A$ be a finite alphabet, let $S$ be the left shift on
$\mathcal A^{\mathbb Z}$, and, for $x\in\mathcal A^{\mathbb Z}$, put
$X_x:=\overline{\{S^ax:a\in\mathbb Z\}}$.

\begin{corollary}\label{cor:symbolic}
Suppose that $(X_x,S^k)$ is minimal for some $k\in\mathbb N$. Let
$\alpha\in\mathcal A$ occur in $x$, and let
$p_1,\ldots,p_d\in\mathbb Z[t]$ satisfy $p_i(0)=0$. Then, for every
$j\in\mathbb Z$, the set
\[
\begin{aligned}
Q_{\alpha,k,j}
:=
\bigl\{
q\in\mathbb Z:
&\text{ there exists }a\in\mathbb Z\text{ such that}\\
&x(a)=\alpha,\quad\text{and}\\
&x(a+p_i(kq+j))=\alpha\quad(1\leq i\leq d)
\bigr\}
\end{aligned}
\]
is syndetic. For every $q\in Q_{\alpha,k,j}$, the integers $a$ realizing
the configuration form a syndetic set.
\end{corollary}

\begin{proof}
Let $U_\alpha:=\{y\in X_x:y(0)=\alpha\}$. By
Corollary~\ref{cor:return-times}, the set of $q$ for which
\[
U_\alpha\cap\bigcap_{i=1}^dS^{-p_i(kq+j)}U_\alpha
\]
is nonempty is syndetic. Since the orbit of $x$ is dense  in $X_x$, this set is
exactly $Q_{\alpha,k,j}$. For each $q\in Q_{\alpha,k,j}$, minimality of
$S^k$ implies that $S$ is minimal, so the entrance times of $x$ into the
displayed open set are syndetic. 
\end{proof}

We use the terminology of \cite{GHSWY2025}. One-sided
sequences are indexed by $\mathbb N=\{1,2,\ldots\}$ and
$(Sx)(r):=x(r+1)$. A partition
$\mathbb N=N_1\cup\cdots\cup N_s$ is \emph{irreducible of type $k$},
where $k\geq2$, if there exist an index $i$ and a nonzero uniformly
recurrent sequence $\omega\in\{0,1\}^{\mathbb N}$ in the forward orbit
closure of $\mathbf1_{N_i}$ such that every finite word occurring in
$\omega$ has an occurrence beginning at a position of the form $kn+1$ for some
$n\in\mathbb N$.

\begin{lemma}\label{lem:dense-power-orbit-minimal}
Let $R:Y\to Y$ be a minimal continuous surjection of a compact metric
space, and let $k\geq1$. If some $y\in Y$ has dense forward $R^k$-orbit,
then $R^k$ is minimal.
\end{lemma}

\begin{proof}
Let $M$ be a minimal nonempty closed $R^k$-invariant subset. By minimality
of $R$, the sets $R^rM$, $0\leq r<k$, cover $Y$ and, after repetitions are
removed, form a pairwise disjoint clopen partition into
$R^k$-invariant sets. The forward $R^k$-orbit of $y$ remains in one member
of this partition and is dense in $Y$, so the partition has one member and
$M=Y$.
\end{proof}

\begin{lemma}\label{lem:natural-extension-minimal}
Let $R:Y\to Y$ be a minimal continuous surjection of a compact metric space.
Then its natural extension $(\widetilde Y,\widetilde R)$ is minimal.
\end{lemma}

\begin{corollary}\label{cor:irreducible-partition}
Let $\mathbb N=N_1\cup\cdots\cup N_s$ be irreducible of type $k$, with
witness pair $(i,\omega)$. Let $p_1,\ldots,p_d\in\mathbb Z[t]$ satisfy
$p_i(0)=0$. For every $0\leq j<k$, the set
\[
\begin{aligned}
\bigl\{
q\in\mathbb N:
&\text{ there exists }a\in\mathbb N\text{ such that}\\
&a,
 a+p_1(kq+j),\ldots,a+p_d(kq+j)
 \in N_i
\bigr\}
\end{aligned}
\]
has bounded gaps in $\mathbb N$.
\end{corollary}

\begin{proof}
Let $X_\omega$ be the one-sided orbit closure of $\omega$. Uniform
recurrence implies that $(X_\omega,S)$ is minimal, while the defining
occurrence condition implies that the forward $S^k$-orbit of $\omega$ is
dense. Hence, by Lemma~\ref{lem:dense-power-orbit-minimal},
$(X_\omega,S^k)$ is minimal. Let $(\widetilde X_\omega,\widetilde S)$ be the natural extension of
$(X_\omega,S)$, realized as a two-sided subshift with the same finite
language as $X_\omega$. In inverse-limit coordinates, the map
\[
(x_0,x_{-1},x_{-2},\ldots)
\longmapsto
(x_0,x_{-k},x_{-2k},\ldots)
\]
conjugates $(\widetilde X_\omega,\widetilde S^k)$ with the natural
extension of $(X_\omega,S^k)$. Hence
$(\widetilde X_\omega,\widetilde S^k)$ is minimal by
Lemma~\ref{lem:natural-extension-minimal}. Choose
$\widetilde\omega\in\widetilde X_\omega$ whose one-sided projection is
$\omega$.

Apply Corollary~\ref{cor:symbolic} to $\widetilde\omega$ and the symbol
$1$. It follows that there is a syndetic set $Q\subseteq\mathbb Z$ such
that, for every $q\in Q$, the common finite language contains a block whose
entries at a translate of
\[
\{0,p_1(kq+j),\ldots,p_d(kq+j)\}
\]
are all $1$. For a fixed $q\in Q\cap\mathbb N$, uniform recurrence allows us to choose
an occurrence of this block in $\omega$ sufficiently far to the right
that all the corresponding coordinates are positive. Choose a prefix of
$\omega$ containing the entire occurrence. Since $\omega$ belongs to the
forward orbit closure of $\mathbf1_{N_i}$, this prefix occurs in
$\mathbf1_{N_i}$ at some nonnegative translate. It therefore yields a
witness $a\in\mathbb N$ for this $q$.

Since $Q$ is syndetic in $\mathbb Z$, the set
$Q\cap\mathbb N$ has bounded gaps in $\mathbb N$.
\end{proof}

For the linear family $p_\ell(t)=\ell t$, the preceding corollary recovers
the coloring consequence following \cite[Theorem~D]{GHSWY2025}. More
generally, the same cell $N_i$ works for every finite polynomial family,
although the gap bound may depend on the family and the residue class.

\section{A counterexample to Leibman's conjectured lower central series identity}
\label{sec:counterexample}

We prove Theorem~\ref{thm:counterexample} by constructing a disconnected
nilpotent Lie group $G$ of step three whose identity component
$N=G^0$ has step two. For the family $p_i(n)=i\,n$, $1\leq i\leq4$, the
associated subgroup satisfies
\[
\widehat H_2\subsetneq\widehat H\cap N_2^4.
\]
Thus the identity asserted in
\cite[Conjecture~11.4]{Leibman2010} fails at level $2$ of the lower
central series.

We denote the lower central series of $N$, $G$, and $\widehat H$ by
$(N_r)$, $(G_r)$, and $(\widehat H_r)$, respectively.

\subsection{Construction of the nilsystem}

Let $N$ be the Lie group with underlying manifold $\mathbb R^5$, coordinates
$(x,y,z,u,w)$, and multiplication
\begin{equation}\label{eq:N-law}
(x,y,z,u,w)(x',y',z',u',w')
=
(x+x',y+y',z+z'+xy',u+u',w+w').
\end{equation}
Thus $N$ is the direct product of the real Heisenberg group and two central
copies of $\mathbb R$. Let $X,Y,Z,U,W$ denote the corresponding basis
vectors of $\mathfrak n:=\operatorname{Lie}(N)$.
Up to antisymmetry, the only nonzero bracket among
these basis vectors is $[X,Y]=Z$. Hence
$N_2=[N,N]=\exp(\mathbb R Z)$ and $N_3=\{e\}$.

Let $\Lambda:=\mathbb Z^5$. This is the product of the standard integer
Heisenberg lattice with $\mathbb Z^2$, and hence is a lattice in $N$.
Define $B:N\to N$ by
\begin{equation}\label{eq:def-B}
B(x,y,z,u,w):=(x,y,z+u,u+w,w).
\end{equation}

\begin{lemma}\label{lem:B}
The map $B$ is a unipotent automorphism of the Lie group $N$ satisfying
$B(\Lambda)=\Lambda$. Its differential is given by
\[
\begin{aligned}
B_*X&=X, & B_*Y&=Y, & B_*Z&=Z,\\
B_*U&=U+Z, & B_*W&=W+U.
\end{aligned}
\]
\end{lemma}

\begin{proof}
Substitution into \eqref{eq:N-law} shows that $B(gg')=B(g)B(g')$. Its
inverse is
\[
B^{-1}(x,y,z,u,w)
=
(x,y,z-u+w,u-w,w).
\]
Both $B$ and $B^{-1}$ have coordinate formulas with integer coefficients,
so
$B(\Lambda)=\Lambda$. The differential formulas follow from
\eqref{eq:def-B}. In particular,
$(B_*-\operatorname{id}_{\mathfrak n})W=U$,
$(B_*-\operatorname{id}_{\mathfrak n})U=Z$, and
$(B_*-\operatorname{id}_{\mathfrak n})$ vanishes on
$\operatorname{span}_{\mathbb R}\{X,Y,Z\}$. Thus
$(B_*-\operatorname{id}_{\mathfrak n})^3=0$.
\end{proof}

Form the semidirect product
\[
G:=N\rtimes_{B^{-1}}\mathbb Z,
\]
with multiplication
\[
(g,m)(h,n)
:=
\bigl(gB^{-m}(h),m+n\bigr).
\]
Let $\gamma:=(e,1)$. Then
\begin{equation}\label{eq:gamma-relation}
\gamma^{-1}g\gamma=B(g)
\qquad
(g\in N).
\end{equation}
Set
\[
\Gamma:=\Lambda\rtimes_{B^{-1}}\mathbb Z
=
\langle\Lambda,\gamma\rangle.
\]

\begin{proposition}\label{prop:G-structure}
The group $G$ is nilpotent of step three, and
\[
G^0=N,
\qquad
G_2=\exp(\mathbb R U+\mathbb R Z),
\qquad
G_3=\exp(\mathbb R Z),
\qquad
G_4=\{e\}.
\]
Moreover, $\Gamma$ is a lattice in $G$, and the natural map
\begin{equation}\label{eq:N-quotient-G-quotient}
N/\Lambda\longrightarrow G/\Gamma,
\qquad
g\Lambda\longmapsto g\Gamma,
\end{equation}
is a homeomorphism. In particular, $G/\Gamma$ is connected.
\end{proposition}

\begin{proof}
Put
\[
K_U:=\exp(\mathbb R U+\mathbb R Z),
\qquad
K_Z:=\exp(\mathbb R Z).
\]
Both $K_U$ and $K_Z$ are central subgroups of $N$
and invariant under $B$, hence normal in $G$.
Using \eqref{eq:gamma-relation} and the commutator
convention $[g,h]=g^{-1}h^{-1}gh$, we obtain
\[
\begin{aligned}
[\exp(tW),\gamma]&=\exp(tU),\\
[\exp(tU),\gamma]&=\exp(tZ),\\
[\exp(sX),\exp(tY)]&=\exp(stZ).
\end{aligned}
\]
It follows that $K_U\subseteq G_2$ and $K_Z\subseteq G_3$. Conversely,
$[N,N]=K_Z\subseteq K_U$, and $B$ induces the identity on $N/K_U$.
Thus $G/K_U$ is abelian and $G_2=K_U$.
Since $K_U$ is central in $N$ and $B$ acts trivially
on $K_U/K_Z$, $[G,K_U]\subseteq K_Z$. The reverse inclusion follows from
the second displayed commutator, so $G_3=K_Z$.
Finally, $B$ fixes $K_Z$ pointwise, so $K_Z$ is central in $G$; hence
$G_4=\{e\}$.

As a topological space, $G=N\times\mathbb Z$, and its connected components
are the sets $N\gamma^m$, $m\in\mathbb Z$. Hence $G^0=N$. Under the same
identification, $\Gamma=\Lambda\times\mathbb Z$ is discrete and closed.
Every element of $G$ has the form $g\gamma^m$ and represents the same right
$\Gamma$-coset as $g\in N$.
Since $N\cap\Gamma=\Lambda$, two elements of $N$
represent the same point of $G/\Gamma$ exactly when they lie in the same
right $\Lambda$-coset.
Thus \eqref{eq:N-quotient-G-quotient} is a
continuous bijection from the compact space $N/\Lambda$ to the Hausdorff
space $G/\Gamma$, and hence is a homeomorphism. It follows in particular
that $\Gamma$ is cocompact.
\end{proof}

Fix a transcendental number $\xi\in\mathbb R$, set
\[
\alpha_0:=(\xi,\xi^2,0,0,\xi^3)\in N,
\qquad
a:=\alpha_0\gamma^{-1}\in G,
\]
and let $T_a$ be the corresponding nilrotation on $G/\Gamma$. Under the
homeomorphism~\eqref{eq:N-quotient-G-quotient}, $T_a$ is conjugate to the
affine transformation
\begin{equation}\label{eq:F-counterexample}
F(g\Lambda):=\alpha_0B(g)\Lambda.
\end{equation}
Indeed,
$a g\Gamma=\alpha_0B(g)\Gamma$. Moreover,
$\langle G^0,a\rangle=G$, because
$\gamma^{-1}=\alpha_0^{-1}a$.

\begin{proposition}\label{prop:total-minimal}
The nilrotation $T_a$, equivalently the affine transformation $F$, is
totally minimal.
\end{proposition}

\begin{proof}
Fix $q\geq1$. Since $G/\Gamma$ is connected, $G^0=N$,
$G^0\cap\Gamma=\Lambda$, and $[G^0,G^0]=N_2$, the maximal factor
torus in \cite[Theorem~C]{Leibman2005Zd} is
\[
G^0/\bigl((G^0\cap\Gamma)[G^0,G^0]\bigr)
\cong
N/(N_2\Lambda)
\cong
\mathbb T^4.
\]
Applying the criterion to $n\mapsto a^{qn}$, it suffices to prove that
every orbit of the transformation induced by $F^q$ on this torus is
dense.
The transformation induced by $F$ on this torus is
\[
(x,y,u,w)
\longmapsto
(x+\xi,y+\xi^2,u+w,w+\xi^3).
\]
Fix an initial point $(x_0,y_0,u_0,w_0)$.
After $n$ iterations of $F^q$, its coordinates are
\[
\begin{aligned}
x_n&=x_0+qn\xi,
&
 y_n&=y_0+qn\xi^2,\\
w_n&=w_0+qn\xi^3,
&
 u_n&=u_0+qn w_0+\binom{qn}{2}\xi^3.
\end{aligned}
\]
Let $(m_x,m_y,m_u,m_w)\in\mathbb Z^4\setminus\{0\}$ define a nontrivial
character. If $m_u\neq0$, the associated phase polynomial has an irrational
quadratic coefficient $m_uq^2\xi^3/2$. If $m_u=0$, the coefficient of $n$
is
\[
q(m_x\xi+m_y\xi^2+m_w\xi^3),
\]
which is irrational because $\xi$ is transcendental and
$(m_x,m_y,m_w)\neq(0,0,0)$. 
By Proposition~\ref{prop:weyl-polynomial-criterion}, every orbit of
$F^q$ on the factor torus is dense. Leibman's criterion therefore
implies that $F^q$ is minimal. Since $q\geq1$ was arbitrary, $F$ is
totally minimal.
\end{proof}

\subsection{Leibman's subgroup and its commutator subgroup}
Let
\[
\boldsymbol\sigma
:=
(\sigma_1,\sigma_2,\sigma_3,\sigma_4)
:=
(1,2,3,4),
\qquad
p_i(n):=\sigma_i n.
\]
For $j\geq0$, write
$\boldsymbol\sigma^j:=(\sigma_1^j,\ldots,\sigma_4^j)$, and put
$\mathbf 1:=\boldsymbol\sigma^0$.
Let
\[
\delta_N:N\longrightarrow N^4,
\qquad
\delta_N(x):=(x,x,x,x),
\]
so that $\Delta_N^4=\delta_N(N)$.
With $\gamma=(e,1)$ fixed, define, for $\alpha\in N$ and
$r\in\mathbb Z$,
\[
g_\alpha(r):=(\alpha\gamma^{-1})^r\gamma^r.
\]
Its image under the quotient homomorphism
$G\to G/N\cong\mathbb Z$ is zero, so $g_\alpha(r)\in N$.
The subgroup relevant to \eqref{eq:intro-lower-central-identity} is
\begin{equation}\label{eq:counterexample-Hhat}
\widehat H
:=
\left\langle
\Delta_N^4,
\bigl(g_\alpha(\sigma_1n),\ldots,g_\alpha(\sigma_4n)\bigr):
\alpha\in N,\ n\in\mathbb Z
\right\rangle
\leq N^4.
\end{equation}

For the fixed decomposition $a=\alpha_0\gamma^{-1}$, this is precisely
the subgroup in \eqref{eq:intro-Hhat} appearing in
\cite[Conjecture~11.4]{Leibman2010}, with ambient group $N=G^0$.
No topological closure is taken. Thus $\widehat H_2$ is the internal
commutator subgroup of $\widehat H$, and the identity at level $2$ is
\[
\widehat H_2=\widehat H\cap N_2^4.
\]

The relation \eqref{eq:gamma-relation} gives
\begin{equation}\label{eq:g-alpha-recurrence}
g_\alpha(r+1)=g_\alpha(r)B^r(\alpha),
\qquad
g_\alpha(0)=e,
\end{equation}
for every $r\in\mathbb Z$.

Let $\mathrm x,\mathrm y:N\to\mathbb R$ denote the first two coordinate
homomorphisms. For
$\mathbf g=(g_1,\ldots,g_4)\in N^4$, write
\[
\mathbf x(\mathbf g)
:=
(\mathrm x(g_1),\ldots,\mathrm x(g_4)),
\qquad
\mathbf y(\mathbf g)
:=
(\mathrm y(g_1),\ldots,\mathrm y(g_4)).
\]
Define 
\[
\iota_Z:(\mathbb R^4,+)\longrightarrow N_2^4,
\qquad
\iota_Z(z_1,\ldots,z_4)
:=
\bigl(\exp(z_1Z),\ldots,\exp(z_4Z)\bigr).
\]
For a subspace $E\leq\mathbb R^4$, let
$\iota_Z(E):=\{\iota_Z(z):z\in E\}$.

\begin{lemma}\label{lem:H2-exact}
One has
\[
\widehat H_2
=
\iota_Z\bigl(
\operatorname{span}_{\mathbb R}
\{\mathbf 1,\boldsymbol\sigma,\boldsymbol\sigma^2\}
\bigr).
\]
\end{lemma}

\begin{proof}
Put
\[
V:=\operatorname{span}_{\mathbb R}
\{\mathbf 1,\boldsymbol\sigma\}.
\]
For every $\mathbf g\in\Delta_N^4$, both
$\mathbf x(\mathbf g)$ and $\mathbf y(\mathbf g)$ lie in
$\mathbb R\mathbf 1$. For
\[
\mathbf g_{\alpha,n}
:=
\bigl(g_\alpha(\sigma_i n)\bigr)_{i=1}^4,
\]
the fact that $B$ fixes the $x$ and $y$ coordinates, together with
\eqref{eq:g-alpha-recurrence}, gives
\[
\mathbf x(\mathbf g_{\alpha,n})
=
n\,\mathrm x(\alpha)\boldsymbol\sigma,
\qquad
\mathbf y(\mathbf g_{\alpha,n})
=
n\,\mathrm y(\alpha)\boldsymbol\sigma.
\]
Since $\mathbf x$ and $\mathbf y$ are homomorphisms, it follows that
\[
\mathbf x(\mathbf g),\mathbf y(\mathbf g)\in V
\qquad
(\mathbf g\in\widehat H).
\]

For $\mathbf g,\mathbf h\in\widehat H$, the commutator formula in $N^4$
gives
\[
[\mathbf g,\mathbf h]
=
\iota_Z\bigl(
\mathbf x(\mathbf g)\odot\mathbf y(\mathbf h)
-
\mathbf x(\mathbf h)\odot\mathbf y(\mathbf g)
\bigr),
\]
where $\odot$ denotes coordinatewise multiplication. Since
\[
\operatorname{span}_{\mathbb R}
\{u\odot v:u,v\in V\}
=
\operatorname{span}_{\mathbb R}
\{\mathbf 1,\boldsymbol\sigma,\boldsymbol\sigma^2\},
\]
we obtain
\[
\widehat H_2
\subseteq
\iota_Z\bigl(
\operatorname{span}_{\mathbb R}
\{\mathbf 1,\boldsymbol\sigma,\boldsymbol\sigma^2\}
\bigr).
\]

For the reverse inclusion, define, for $a,b\in\mathbb R$,
\[
\begin{aligned}
D_X(a)&:=\delta_N(\exp(aX)),
&
D_Y(b)&:=\delta_N(\exp(bY)),\\
C_X(a)&:=\bigl(g_{\exp(aX)}(\sigma_i)\bigr)_{i=1}^4,
&
C_Y(b)&:=\bigl(g_{\exp(bY)}(\sigma_i)\bigr)_{i=1}^4.
\end{aligned}
\]
These elements belong to $\widehat H$. Since $B_*X=X$ and $B_*Y=Y$,
\eqref{eq:g-alpha-recurrence} gives
\[
C_X(a)=\bigl(\exp(a\sigma_iX)\bigr)_{i=1}^4,
\qquad
C_Y(b)=\bigl(\exp(b\sigma_iY)\bigr)_{i=1}^4.
\]
Taking coordinatewise commutators yields
\[
\begin{aligned}
[D_X(a),D_Y(b)]
&=\iota_Z(ab\mathbf 1),\\
[D_X(a),C_Y(b)]
&=\iota_Z(ab\boldsymbol\sigma),\\
[C_X(a),C_Y(b)]
&=\iota_Z(ab\boldsymbol\sigma^2).
\end{aligned}
\]
Taking $b=1$ and varying $a$ gives the images of all three real lines.
Since these subgroups are central, their products give the image of
their linear span, proving the reverse inclusion.
\end{proof}

\subsection{The additional cubic direction}

For $t\in\mathbb R$, put $\alpha_t:=\exp(tW)$. Since $Z,U,W$ commute and
\[
B_*^jW=W+jU+\binom{j}{2}Z,
\]
we have
\[
B^j(\alpha_t)
=
\exp\bigl(t(W+jU+\binom{j}{2}Z)\bigr)
\qquad (j\in\mathbb Z).
\]

For $r\in\mathbb Z$ and $j\geq0$, we use the polynomial convention
\[
\binom{r}{j}
:=
\frac{r(r-1)\cdots(r-j+1)}{j!}.
\]

For $r\in\mathbb Z$, put
\[
R_t(r):=
\exp\bigl(t(rW+\binom{r}{2}U+\binom{r}{3}Z)\bigr).
\]
The identities
\[
\binom{r+1}{2}-\binom{r}{2}=r,
\qquad
\binom{r+1}{3}-\binom{r}{3}=\binom{r}{2}
\]
give
\[
R_t(0)=e,
\qquad
R_t(r+1)=R_t(r)B^r(\alpha_t)
\qquad (r\in\mathbb Z).
\]
Since this two-sided recurrence is uniquely determined by its value at
$r=0$, comparison with \eqref{eq:g-alpha-recurrence} gives
\begin{equation}\label{eq:g-alpha-W}
g_{\alpha_t}(r)
=
R_t(r)
=
\exp\bigl(t(rW+\binom{r}{2}U+\binom{r}{3}Z)\bigr)
\qquad (r\in\mathbb Z).
\end{equation}

For $t\in\mathbb R$ and $n\in\mathbb Z$, define
\[
\mathbf c_t(n)
:=
\bigl(g_{\alpha_t}(\sigma_i n)\bigr)_{i=1}^4
\in\widehat H.
\]

\begin{lemma}\label{lem:cubic-witness}
For every $t\in\mathbb R$,
\begin{equation}\label{eq:cubic-witness}
\mathbf c_{3t}(1)\mathbf c_{-3t}(2)\mathbf c_t(3)
=
\iota_Z(t\boldsymbol\sigma^3).
\end{equation}
Consequently,
$\iota_Z(\mathbb R\boldsymbol\sigma^3)
\subseteq\widehat H\cap N_2^4$.
\end{lemma}

\begin{proof}
All three factors in \eqref{eq:cubic-witness} lie in the abelian subgroup
$\exp(\mathbb R Z+\mathbb R U+\mathbb R W)^4$, so their logarithms add.
Fix $1\leq i\leq4$ and put $r:=\sigma_i$.
By \eqref{eq:g-alpha-W}, the logarithms of the three factors in
coordinate $i$ are
\[
\begin{aligned}
&3t\bigl(rW+\binom{r}{2}U+\binom{r}{3}Z\bigr),\\
&-3t\bigl(2rW+\binom{2r}{2}U+\binom{2r}{3}Z\bigr),\\
&t\bigl(3rW+\binom{3r}{2}U+\binom{3r}{3}Z\bigr).
\end{aligned}
\]
In this sum, the coefficients of $W$ and $U$ vanish, while
\[
3\binom r3-3\binom{2r}{3}+\binom{3r}{3}=r^3.
\]
Thus the coordinate indexed by $i$ is
$\exp(t\sigma_i^3Z)$, proving \eqref{eq:cubic-witness}.
\end{proof}
\begin{proposition}\label{prop:H-final}
For the group $\widehat H$ in \eqref{eq:counterexample-Hhat},
\[
\widehat H\cap N_2^4=N_2^4.
\]
Consequently,
\[
\widehat H_2
\subsetneq
\widehat H\cap N_2^4,
\qquad
(\widehat H\cap N_2^4)/\widehat H_2
\cong
(\mathbb R,+).
\]
\end{proposition}

\begin{proof}
The Vandermonde
determinant
\[
\det(\sigma_i^{j-1})_{1\leq i,j\leq4}
=
\prod_{1\leq i<j\leq4}(\sigma_j-\sigma_i)
=12
\]
is nonzero. Hence
$\mathbf 1,\boldsymbol\sigma,\boldsymbol\sigma^2,
\boldsymbol\sigma^3$ form a basis of $\mathbb R^4$.
Lemma~\ref{lem:H2-exact} gives
\[
\iota_Z\bigl(
\mathbb R\mathbf 1+
\mathbb R\boldsymbol\sigma+
\mathbb R\boldsymbol\sigma^2
\bigr)
\subseteq\widehat H,
\]
while Lemma~\ref{lem:cubic-witness} gives
\[
\iota_Z(\mathbb R\boldsymbol\sigma^3)\subseteq\widehat H.
\]
Therefore
$N_2^4=\iota_Z(\mathbb R^4)\subseteq\widehat H$, and hence
$\widehat H\cap N_2^4=N_2^4$.

Let
\[
E:=
\operatorname{span}_{\mathbb R}
\{\mathbf 1,\boldsymbol\sigma,\boldsymbol\sigma^2\}.
\]
Since $\iota_Z$ is a topological group isomorphism,
$\widehat H_2=\iota_Z(E)$,
$\widehat H\cap N_2^4=\iota_Z(\mathbb R^4)$, and $E$ is a closed
subspace of codimension one, we have
\[
(\widehat H\cap N_2^4)/\widehat H_2
\cong
\mathbb R^4/E
\cong
(\mathbb R,+)
\]
as topological groups.
\end{proof}

\begin{proof}[Proof of Theorem~\ref{thm:counterexample}]
By the construction of $N$ and
Proposition~\ref{prop:G-structure}, $G$ is nilpotent of step three,
$\Gamma$ is a lattice,
$G^0=N\cong H_3(\mathbb R)\times\mathbb R^2$, and $G/\Gamma$ is
connected.
Moreover,
$\gamma^{-1}=\alpha_0^{-1}a$ implies
$G=\langle G^0,a\rangle$, while
Proposition~\ref{prop:total-minimal} gives total minimality.
For $p_i(n)=i\,n$, the subgroup in
\eqref{eq:counterexample-Hhat} is exactly that in
\eqref{eq:intro-Hhat}. Proposition~\ref{prop:H-final} now gives the
required strict containment and quotient isomorphism.
\end{proof}

\begin{remark}
The parameter $\xi$ is used only to prove total minimality; the subgroup
calculation depends only on $B$ and $\boldsymbol\sigma$.
\end{remark}

\appendix

\section{A two-step model for the central induction}
\label{sec:two-step-model}

This appendix illustrates the central induction in the first
noncommutative case for the family
\[
p_1(n)=n,\qquad p_2(n)=2n,\qquad p_3(n)=n^2.
\]
It is not used in the proof of Theorem~\ref{thm:main}.

Throughout this section, let $N$ be a connected, simply connected nilpotent
Lie group of step at most two, let $\Lambda\leq N$ be a lattice, and let
$A\in\operatorname{Aut}(N)$ be a unipotent automorphism satisfying
$A(\Lambda)=\Lambda$. Fix $b\in N$ and assume that the affine transformation
\begin{equation}\label{eq:model-affine-map}
  F:N/\Lambda\longrightarrow N/\Lambda,
  \qquad F(x\Lambda)=bA(x)\Lambda,
\end{equation}
is minimal. We do not assume that $A=\operatorname{id}_N$.

For $n\in\mathbb Z$, put
\[
\Sigma_n:=
\left\{
\bigl(F^nx,F^{2n}x,F^{n^2}x\bigr):
x\in N/\Lambda
\right\}.
\]
Then
\begin{equation}\label{eq:model-orbit-closure}
C:=\overline{\bigcup_{n\in\mathbb Z}\Sigma_n}.
\end{equation}

\begin{proposition}\label{prop:two-step-model}
Under the preceding hypotheses, $C$ is connected.
\end{proposition}

We use the suspension notation of Section~\ref{sec:suspension}. Thus
\[
\begin{aligned}
\mathfrak n&=\operatorname{Lie}(N),&
\mathcal D&=\log A_*,\\
\beta(u)&=\alpha^u\tau^{-u},&
v&=\log\alpha-\log\tau\in\mathfrak n.
\end{aligned}
\]
For this family, define
\begin{align}
  g(t)&:=\bigl(\beta(t),\beta(2t),\beta(t^2)\bigr),
  \label{eq:model-g}\\
  E_t(Y)&:=\bigl(A_*^tY,A_*^{2t}Y,A_*^{t^2}Y\bigr),
  \label{eq:model-E}\\
  \Psi_t(Y)&:=\bigl(A_*^tY,2A_*^{2t}Y,
                    2tA_*^{t^2}Y\bigr).
  \label{eq:model-Psi}
\end{align}

For $n\in\mathbb Z$ and $x\in N$, \eqref{eq:F-interpolation} implies
\[
\bigl(
F^n(x\Lambda),F^{2n}(x\Lambda),F^{n^2}(x\Lambda)
\bigr)
=
g(n)\bigl(
A^n(x),A^{2n}(x),A^{n^2}(x)
\bigr)\Lambda^3.
\]
In particular,
\[
g(n)\Lambda^3\in\Sigma_n.
\]
Moreover, for $Y\in\mathfrak n$, the curve in $\Sigma_n$ obtained by
setting $x=\exp(sY)\Lambda$ has the lift
\[
s\longmapsto g(n)\exp\bigl(sE_n(Y)\bigr).
\]

By Lemma~\ref{lem:interpolation},
\begin{equation}\label{eq:model-log-derivative}
  g(t)^{-1}g'(t)=\Psi_t(v).
\end{equation}
For each fixed $Y\in\mathfrak n$, both $t\mapsto E_t(Y)$ and
$t\mapsto\Psi_t(Y)$ are ordinary vector-valued polynomials.

Differentiating $E_t$ gives
\begin{equation}\label{eq:model-E-derivative}
\frac d{dt}E_t(Y)=\Psi_t(\mathcal DY).
\end{equation}

By Theorem~\ref{thm:Leibman-orbit-closure} and the zero-fiber argument in
Section~\ref{sec:components}, there exist a connected rational subgroup
$H_0\leq N^3$ and finitely many $y_0,\ldots,y_{r-1}\in N^3$ such that
\begin{equation}\label{eq:model-finite-components}
  C=\bigsqcup_{\nu=0}^{r-1}H_0y_\nu\Lambda^3/\Lambda^3.
\end{equation}
After reindexing, the component containing the diagonal is
\begin{equation}\label{eq:model-C0}
  C_0=H_0\Lambda^3/\Lambda^3,
  \qquad \Delta_N^3\subseteq H_0.
\end{equation}
Write
\[
  \mathfrak h_0=\operatorname{Lie}(H_0).
\]
It remains to prove that every slice in
\eqref{eq:model-orbit-closure} lies in $C_0$.

\subsection{The abelian base case}

We first record the argument when $N$ is abelian. This will be applied to
the horizontal quotient of the two-step group.

\begin{lemma}\label{lem:model-abelian}
Assume, in addition to the preceding hypotheses, that $N$ is abelian. Then
\begin{equation}\label{eq:model-abelian-strong}
  g(t)\in H_0 \qquad (t\in\mathbb R),
\end{equation}
and consequently $C=C_0$ is connected.
\end{lemma}

\begin{proof}
Each $\Sigma_n$ is a connected embedded copy of $N/\Lambda$ and lies in
one of the components in \eqref{eq:model-finite-components}. 

For $Y\in\mathfrak n$, the left and right logarithmic derivatives of the
displayed lifted curve at $s=0$ are both $E_n(Y)$, since $N$ is abelian.
Under the same local lift, either
trivialization identifies the tangent space at $g(n)$ of the lifted
$H_0$-orbit with $\mathfrak h_0$. Hence
\begin{equation}\label{eq:model-abelian-E-in-l-integers}
  E_n(\mathfrak n)\subseteq\mathfrak h_0
  \qquad (n\in\mathbb Z).
\end{equation}
By Lemma~\ref{lem:polynomial-identity},
\begin{equation}\label{eq:model-abelian-E-in-l}
  E_t(\mathfrak n)\subseteq\mathfrak h_0
  \qquad (t\in\mathbb R).
\end{equation}

Because $N^3$ is abelian, the quotient group
\begin{equation}\label{eq:model-abelian-K}
  K:=N^3/(H_0\Lambda^3)
\end{equation}
is a compact connected torus. Let $\chi\in\widehat K$. By
Lemma~\ref{lem:character-lift}, the pullback of $\chi$ along the quotient
homomorphism $N^3\to K$ has a unique continuous lift
\[
  \widetilde f_\chi:N^3\longrightarrow\mathbb R,
  \qquad \widetilde f_\chi(e)=0,
\]
and this lift is a group homomorphism. Let
\[
  \ell_\chi=d_e\widetilde f_\chi\in(\mathfrak n^3)^*.
\]
Then $\ell_\chi$ is rational, annihilates $\mathfrak h_0$, and is integral on
$\log\Lambda^3$. Define
\begin{equation}\label{eq:model-abelian-Q-eta}
  Q_\chi(t):=\widetilde f_\chi(g(t)),
  \qquad
  \eta_{\chi,t}:=\ell_\chi\circ\Psi_t\in\mathfrak n^*.
\end{equation}
By \eqref{eq:model-log-derivative},
\begin{equation}\label{eq:model-abelian-Q-derivative}
  Q_\chi'(t)=\eta_{\chi,t}(v).
\end{equation}
The right-hand side is polynomial in $t$, and $Q_\chi(0)=0$, so $Q_\chi$ is
a real polynomial.

The set
\begin{equation}\label{eq:model-abelian-finite-phases}
  \left\{\exp\bigl(2\pi iQ_\chi(n)\bigr):n\in\mathbb Z\right\}
\end{equation}
is finite. Indeed, the induced quotient map $N^3/\Lambda^3\to K$ is constant
on every component
$H_0y_\nu\Lambda^3/\Lambda^3$, and there are only finitely many components.
By Lemma~\ref{lem:finite-phases},
\begin{equation}\label{eq:model-abelian-rational-derivative}
  \eta_{\chi,k}(v)=Q_\chi'(k)\in\mathbb Q
  \qquad (k\in\mathbb Z).
\end{equation}

Fix $k\in\mathbb Z$. Since $\ell_\chi$ is rational, $A_*$ preserves the Mal'cev rational
structure, and the coefficients $1,2,2k$ in
\eqref{eq:model-Psi} are integers, the functional
$\eta_{\chi,k}$ is rational. It is horizontal because
$N$ is abelian. 
For each fixed $Y\in\mathfrak n$, the curve $t\mapsto E_t(Y)$ takes
values in $\mathfrak h_0$ by
\eqref{eq:model-abelian-E-in-l}. Hence its derivative also lies in
$\mathfrak h_0$. Using \eqref{eq:model-E-derivative}, we obtain
\[
\Psi_t(\mathcal D\mathfrak n)\subseteq\mathfrak h_0.
\]
As $\ell_\chi(\mathfrak h_0)=0$, this gives
\[
\eta_{\chi,t}\circ\mathcal D=0.
\]
Since $A_*=\exp(\mathcal D)$, it follows that
\[
\eta_{\chi,t}\circ A_*=\eta_{\chi,t}.
\]
If $\eta_{\chi,k}$ were nonzero, Lemma~\ref{lem:horizontal-obstruction}
would imply
\[
  \eta_{\chi,k}(v)\notin\mathbb Q,
\]
contradicting \eqref{eq:model-abelian-rational-derivative}. Hence
\begin{equation}\label{eq:model-abelian-eta-integers-zero}
  \eta_{\chi,k}=0
  \qquad (k\in\mathbb Z).
\end{equation}
The map $t\mapsto\eta_{\chi,t}$ is polynomial, so by
Lemma~\ref{lem:polynomial-identity},
\begin{equation}\label{eq:model-abelian-eta-zero}
  \eta_{\chi,t}=0
  \qquad (t\in\mathbb R).
\end{equation}

As $\chi$ ranges over $\widehat K$, the differentials $\ell_\chi$ span the
annihilator of $\mathfrak h_0$ in $(\mathfrak n^3)^*$. Therefore
\eqref{eq:model-abelian-eta-zero}, applied to $v$, implies
\begin{equation}\label{eq:model-abelian-Psi-v-in-l}
  \Psi_t(v)\in\mathfrak h_0
  \qquad (t\in\mathbb R).
\end{equation}
By \eqref{eq:model-log-derivative}, the projection of $g$ to $N^3/H_0$ has
zero left logarithmic derivative. Since $g(0)=e$, this projected path is
constant at the identity coset, proving
\eqref{eq:model-abelian-strong}.

Every $\Sigma_n$ is connected and contains
$g(n)\Lambda^3\in C_0$. Hence $\Sigma_n\subseteq C_0$ for every $n$, and
taking the closure of their union yields $C=C_0$.
\end{proof}

\subsection{Proof in the two-step case}

\begin{proof}[Proof of Proposition~\ref{prop:two-step-model}]
If $N$ is abelian, the conclusion follows from
Lemma~\ref{lem:model-abelian}. We may therefore assume that $N$ has step
exactly two. 
Put
\begin{equation}\label{eq:model-Z}
Z=N_2=[N,N],
\qquad
\mathfrak z=\operatorname{Lie}(Z).
\end{equation}
By Proposition~\ref{prop:malcev-package}(iv), $Z$ is rational, so
\[
\overline\Lambda:=\Lambda Z/Z
\]
is a lattice in $N/Z$. Let
\[
q_N:N^3\longrightarrow (N/Z)^3
\]
be the quotient homomorphism. The group $N/Z$ is connected, simply
connected, and abelian. Since $Z$ is characteristic, $A$ induces a
unipotent automorphism of $N/Z$ preserving $\overline\Lambda$. The
induced affine system on $(N/Z)/\overline\Lambda$ is minimal, being a
factor of \eqref{eq:model-affine-map}.

Let $\overline C$ be the polynomial diagonal orbit closure in the quotient
system, and let $\overline H_0$ be the connected rational subgroup defining
its connected component containing the identity coset.
Since the induced factor map is continuous and $C$ is compact,
$\overline C$ is precisely the image of $C$.
Applying $q_N$ coordinatewise to \eqref{eq:F-interpolation} shows that
the quotient interpolating map is $q_N\circ g$. By
Lemma~\ref{lem:model-abelian},
\begin{equation}\label{eq:model-quotient-strong}
  \overline C
  =\overline H_0\,\overline\Lambda^3/\overline\Lambda^3,
  \qquad
  q_N(g(t))\in\overline H_0
  \qquad (t\in\mathbb R).
\end{equation}

Applying the quotient map to
\eqref{eq:model-finite-components} expresses $\overline C$ as a finite
union of compact $q_N(H_0)$-orbits, each the image of a component of $C$.
After repetitions are removed, distinct orbits are disjoint. Since these
orbits are compact and their union is finite, each is clopen in
$\overline C$. As $\overline C$ is connected, only the orbit containing
the identity remains. Thus
\[
\overline C
=
q_N(H_0)\overline\Lambda^3/\overline\Lambda^3.
\]

The image $q_N(H_0)$ is a connected Lie subgroup of the connected and
simply connected nilpotent group $(N/Z)^3$ and is therefore closed by
Lemma~\ref{lem:connected-subgroups-simply-connected}. The orbit of the
identity coset under $q_N(H_0)$ is closed, so $q_N(H_0)$ is rational
relative to $\overline\Lambda^3$. On the other hand, the connected
component of $\overline C$ containing the identity coset is the
$\overline H_0$-orbit in
\eqref{eq:model-quotient-strong}. Corollary~\ref{cor:identity-orbit-comparison}
therefore implies
\begin{equation}\label{eq:model-qL-barL}
q_N(H_0)=\overline H_0.
\end{equation}
Since $\ker q_N=Z^3$, we have
\[
q_N^{-1}\bigl(q_N(H_0)\bigr)=H_0Z^3.
\]
Together with the second assertion in
\eqref{eq:model-quotient-strong}, this gives
\begin{equation}\label{eq:model-g-in-M}
  g(t)\in M:=H_0Z^3
  \qquad (t\in\mathbb R).
\end{equation}
Since $Z^3$ is central and rational in $N^3$,
Lemma~\ref{lem:rational-torus}(i) yields
\begin{equation}\label{eq:model-M-structure}
  M\text{ is connected and rational},
  \qquad H_0\trianglelefteq M,
  \qquad
  \mathfrak m:=\operatorname{Lie}(M)=\mathfrak h_0+\mathfrak z^3.
\end{equation}

\medskip
\noindent\textit{Tangent directions.}

For $n\in\mathbb Z$, consider again the connected slice $\Sigma_n$. After
lifting $\Sigma_n$ locally through $g(n)$, right trivialization identifies
the tangent space of the lifted slice at $g(n)$ with
\begin{equation}\label{eq:model-slice-tangent}
  \operatorname{Ad}(g(n))E_n(\mathfrak n).
\end{equation}
The slice lies in one $H_0$-orbit. Under the same local lift, right
trivialization identifies the tangent space at $g(n)$ of the lifted
$H_0$-orbit with $\mathfrak h_0$. Hence
\[
  \operatorname{Ad}(g(n))E_n(\mathfrak n)\subseteq\mathfrak h_0.
\]
By \eqref{eq:model-g-in-M}, $g(n)\in M$, and by
\eqref{eq:model-M-structure}, $H_0\trianglelefteq M$. Therefore
$\operatorname{Ad}(g(n))\mathfrak h_0=\mathfrak h_0$, and we conclude that
\begin{equation}\label{eq:model-E-in-l-integers}
  E_n(\mathfrak n)\subseteq\mathfrak h_0
  \qquad (n\in\mathbb Z).
\end{equation}
The polynomial identity principle extends this inclusion to
\begin{equation}\label{eq:model-E-in-l}
  E_t(\mathfrak n)\subseteq\mathfrak h_0
  \qquad (t\in\mathbb R).
\end{equation}

We also record invariance under the diagonal automorphism
$A_\Delta=A^{\times3}$. The homeomorphism $F^{\times3}$ preserves every
slice $\Sigma_n$ and hence preserves $C$. It sends the identity coset to
$b_\Delta\Lambda^3$, where
\[
b_\Delta=(b,b,b)\in\Delta_N^3\subseteq H_0.
\]
Thus $b_\Delta\Lambda^3\in C_0$. Since $F^{\times3}$ permutes the connected
components of $C$, the image $F^{\times3}(C_0)$ is the component containing
$b_\Delta\Lambda^3$, namely $C_0$. Hence
\[
F^{\times3}(C_0)=C_0.
\]
Now
\[
F^{\times3}=\lambda_{b_\Delta}\circ A_\Delta,
\]
and left translation by $b_\Delta\in H_0$ preserves $C_0$. 
Consequently, $A_\Delta(C_0)=C_0$. The full inverse image of $C_0$
under the quotient map $N^3\to N^3/\Lambda^3$ is $H_0\Lambda^3$,
whose identity component is $H_0$ by
Lemma~\ref{lem:malcev}(i). Since $A_\Delta$ preserves both
$\Lambda^3$ and $C_0$, it preserves $H_0\Lambda^3$ and hence its
identity component. Therefore
\begin{equation}\label{eq:model-A-invariance}
A_\Delta(H_0)=H_0,
\qquad
(A_\Delta)_*\mathfrak h_0=\mathfrak h_0,
\qquad
(A_\Delta)_*\mathfrak m=\mathfrak m.
\end{equation}
The last assertion follows because $A(Z)=Z$.

\medskip
\noindent\textit{Polynomial derivative directions.}

The inclusion \eqref{eq:model-g-in-M} implies only that the left logarithmic
derivative $\Psi_t(v)$ lies in $\mathfrak m$. To construct horizontal
functionals, we need the stronger inclusion
$\Psi_t(\mathfrak n)\subseteq\mathfrak m$.

Fix $k\in\mathbb Z$ and define
\begin{equation}\label{eq:model-Sk}
  S_k:=\{Y\in\mathfrak n:\Psi_k(Y)\in\mathfrak m\}.
\end{equation}
We claim that $S_k=\mathfrak n$.

The map $\Psi_k$ is rational, since $A_*$ preserves
$\mathfrak n_{\mathbb Q}$ and $1,2,2k$ are integers. As
$\mathfrak m$ is rational,
\[
S_k=\Psi_k^{-1}(\mathfrak m)
\]
is a rational subspace of $\mathfrak n$. Next,
\begin{equation}\label{eq:model-Psi-A-equivariance}
  \Psi_k(A_*Y)=(A_\Delta)_*\Psi_k(Y),
\end{equation}
so \eqref{eq:model-A-invariance} gives $A_*S_k=S_k$.
Finally, for $X,Y\in\mathfrak n$, coordinatewise preservation of the Lie
bracket yields
\begin{equation}\label{eq:model-bracket-identity}
  [\Psi_k(X),E_k(Y)]=\Psi_k([X,Y]).
\end{equation}
If $X\in S_k$, then $\Psi_k(X)\in\mathfrak m$, while
\eqref{eq:model-E-in-l} gives $E_k(Y)\in\mathfrak h_0$. Since
$H_0\trianglelefteq M$,
\[
  [\mathfrak m,\mathfrak h_0]\subseteq\mathfrak h_0\subseteq\mathfrak m.
\]
Equation~\eqref{eq:model-bracket-identity} therefore shows that
$[X,Y]\in S_k$. Thus $S_k$ is a rational, $A_*$-invariant ideal.

Because $g$ takes values in $M$, its left logarithmic derivative takes values
in $\mathfrak m$. Equation~\eqref{eq:model-log-derivative} yields
\begin{equation}\label{eq:model-v-in-Sk}
  \Psi_k(v)\in\mathfrak m,
  \qquad v\in S_k.
\end{equation}
Suppose that $S_k\neq\mathfrak n$, put $H_k=\exp(S_k)$, and let
$\pi_k:N\to N/H_k$ be the quotient homomorphism. Then $H_k$ is a connected,
closed, normal, rational subgroup. By
Proposition~\ref{prop:malcev-package}(iii), the image $\pi_k(\Lambda)$ is a
lattice in $N/H_k$.
Since $S_k$ is proper, $N/H_k$ has positive dimension; hence
$(N/H_k)/\pi_k(\Lambda)$ is nontrivial.
 Since
$\mathcal D=\log A_*$ is a polynomial in
$A_*-\operatorname{id}_{\mathfrak n}$, the $A_*$-invariance of
$S_k$ implies
\[
  A_*^uS_k=S_k,
  \qquad A^u(H_k)=H_k
  \qquad (u\in\mathbb R).
\]
Using \eqref{eq:model-v-in-Sk} in
$\beta(u)^{-1}\beta'(u)=A_*^uv$, we find that the quotient curve
$u\mapsto\pi_k(\beta(u))$ has zero left logarithmic derivative. By
Lemma~\ref{lem:zero-logarithmic-derivative}, this curve is
constant. Since $\beta(0)=e$,
\begin{equation}\label{eq:model-beta-in-Hk}
  \beta(u)\in H_k \qquad (u\in\mathbb R),
  \qquad b=\beta(1)\in H_k.
\end{equation}
Hence, on the nontrivial compact nilmanifold
\[
  (N/H_k)/\pi_k(\Lambda)
\]
the map induced by $F$ is the automorphism induced by $A$ and fixes the
identity coset. The resulting
system is a factor of the minimal system $(N/\Lambda,F)$ and is therefore
minimal, but a nontrivial minimal system cannot have a fixed point. The
contradiction proves $S_k=\mathfrak n$.
Thus
\begin{equation}\label{eq:model-Psi-in-m-integers}
  \Psi_k(\mathfrak n)\subseteq\mathfrak m
  \qquad (k\in\mathbb Z).
\end{equation}
Applying the polynomial identity principle once more yields
\begin{equation}\label{eq:model-Psi-in-m}
  \Psi_t(\mathfrak n)\subseteq\mathfrak m
  \qquad (t\in\mathbb R).
\end{equation}

\medskip
\noindent\textit{Elimination of the residual central directions.}

It remains to show that the image of $g(t)$ in $M/H_0$ is the identity. Let
\[
\pi_0:M\longrightarrow M/H_0
\]
be the quotient homomorphism, and put
\begin{equation}\label{eq:model-residual-torus}
  \Lambda_M=M\cap\Lambda^3,
  \qquad
  K:=M/(H_0\Lambda_M).
\end{equation}
Since $M$ is a connected subgroup of the connected and simply connected
nilpotent group $N^3$, Lemma~\ref{lem:connected-subgroups-simply-connected}
shows that $M$ is closed and simply connected. Since $M=H_0Z^3$ and $Z^3$
is central, $[M,M]\subseteq H_0$. Lemma~\ref{lem:rational-torus}(ii)
therefore shows that $H_0\Lambda_M$ is a closed normal subgroup of $M$ and
that
\[
K\cong(\mathfrak m/\mathfrak h_0)/\Omega,
\qquad
\Omega:=\log_{M/H_0}\bigl(\pi_0(\Lambda_M)\bigr),
\]
where $\Omega$ is a full lattice. In particular, $K$ is a compact connected
torus.

Fix $\chi\in\widehat K$. By Lemma~\ref{lem:character-lift}, its pullback
to $M$ has a unique continuous homomorphic lift
\[
\widetilde f_\chi:M\longrightarrow\mathbb R,
\qquad
\widetilde f_\chi(e)=0.
\]
Put
\[
\ell_\chi:=d_e\widetilde f_\chi\in\mathfrak m^*.
\]

Then
\begin{equation}\label{eq:model-ell-properties}
  \ell_\chi(\mathfrak h_0)=0,
  \qquad
  \ell_\chi([\mathfrak m,\mathfrak m])=0,
  \qquad
  \ell_\chi(\log\Lambda_M)\subseteq\mathbb Z.
\end{equation}
Since $M$ is rational,
\[
\operatorname{span}_{\mathbb Q}\log\Lambda_M
=
\mathfrak m\cap(\mathfrak n_{\mathbb Q})^3,
\]
and hence $\ell_\chi$ is rational. By \eqref{eq:model-Psi-in-m}, we may define
\begin{equation}\label{eq:model-final-Q-eta}
  Q_\chi(t):=\widetilde f_\chi(g(t)),
  \qquad
  \eta_{\chi,t}:=\ell_\chi\circ\Psi_t\in\mathfrak n^*.
\end{equation}
Equation~\eqref{eq:model-log-derivative} gives
\begin{equation}\label{eq:model-final-Q-derivative}
  Q_\chi'(t)=\eta_{\chi,t}(v).
\end{equation}
The right-hand side is a real polynomial in $t$, and $Q_\chi(0)=0$.
Lemma~\ref{lem:polynomial-antiderivative} therefore shows that $Q_\chi$ is
a real polynomial.

For every $k\in\mathbb Z$, the functional $\eta_{\chi,k}$ is rational. 
Indeed, $A_*^k$ preserves $\mathfrak n_{\mathbb Q}$, and
\eqref{eq:model-Psi-in-m} gives
\[
\Psi_k(\mathfrak n_{\mathbb Q})
\subseteq
\mathfrak m\cap(\mathfrak n_{\mathbb Q})^3.
\]
The rationality of $\ell_\chi$ therefore implies that of
$\eta_{\chi,k}$.
It is
also horizontal. Indeed, by \eqref{eq:model-bracket-identity},
\begin{align*}
  \eta_{\chi,k}([X,Y])
  &=\ell_\chi\bigl(\Psi_k([X,Y])\bigr)\\
  &=\ell_\chi\bigl([\Psi_k(X),E_k(Y)]\bigr)=0,
\end{align*}
because $\Psi_k(X)\in\mathfrak m$, $E_k(Y)\in\mathfrak h_0$, and
$[\mathfrak m,\mathfrak h_0]\subseteq\mathfrak h_0$. Moreover,
differentiating the relation $E_t(Y)\in\mathfrak h_0$ from
\eqref{eq:model-E-in-l} and using \eqref{eq:model-E-derivative} yields
\[
  \Psi_t(\mathcal D\mathfrak n)\subseteq\mathfrak h_0.
\]
Therefore
\[
\eta_{\chi,t}\circ\mathcal D=0.
\]
Since $A_*=\exp(\mathcal D)$, it follows that
\begin{equation}\label{eq:model-final-A-invariant-eta}
\eta_{\chi,t}\circ A_*=\eta_{\chi,t}.
\end{equation}

The phases
\begin{equation}\label{eq:model-final-finite-phases}
  \left\{\exp\bigl(2\pi iQ_\chi(n)\bigr):n\in\mathbb Z\right\}
\end{equation}
form a finite set. To see this, suppose that
$g(n_i)\Lambda^3\in H_0y_\nu\Lambda^3/\Lambda^3$ for $i=1,2$. Write
\[
g(n_i)=h_i y_\nu\lambda_i,
\qquad h_i\in H_0,\quad \lambda_i\in\Lambda^3.
\]
Then
\[
  m_i:=h_i^{-1}g(n_i)=y_\nu\lambda_i\in M
\]
and
\[
  m_1^{-1}m_2=\lambda_1^{-1}\lambda_2\in\Lambda_M.
\]
Thus $g(n_1)$ and $g(n_2)$ have the same image in $K$ and hence determine
the same phase. Since there are only finitely many components, the set in
\eqref{eq:model-final-finite-phases} is finite.

By Lemma~\ref{lem:finite-phases},
\begin{equation}\label{eq:model-final-rational-phase}
  \eta_{\chi,k}(v)=Q_\chi'(k)\in\mathbb Q
  \qquad (k\in\mathbb Z).
\end{equation}
If $\eta_{\chi,k}$ were nonzero,
Lemma~\ref{lem:horizontal-obstruction}, applied using its rationality,
horizontality, and $A_*$-invariance, would imply
\[
  \eta_{\chi,k}(v)\notin\mathbb Q.
\]
This contradicts \eqref{eq:model-final-rational-phase}. Hence
\begin{equation}\label{eq:model-final-eta-integers-zero}
  \eta_{\chi,k}=0
  \qquad (k\in\mathbb Z).
\end{equation}
Since $t\mapsto\eta_{\chi,t}$ is polynomial,
Lemma~\ref{lem:polynomial-identity} yields
\begin{equation}\label{eq:model-final-eta-zero}
  \eta_{\chi,t}=0
  \qquad (t\in\mathbb R).
\end{equation}

By Lemma~\ref{lem:character-lift}, the functionals on
$\mathfrak m/\mathfrak h_0$ induced by the differentials $\ell_\chi$ span
its full dual. Applying \eqref{eq:model-final-eta-zero} to $v$ therefore
yields
\begin{equation}\label{eq:model-final-Psi-v-in-l}
  \Psi_t(v)\in\mathfrak h_0
  \qquad (t\in\mathbb R).
\end{equation}
By \eqref{eq:model-log-derivative} and
\eqref{eq:model-final-Psi-v-in-l}, the projected curve
$\pi_0\circ g$ has zero left logarithmic derivative. Since $g(0)=e$,
Lemma~\ref{lem:zero-logarithmic-derivative} gives
\[
g(t)\in H_0
\qquad (t\in\mathbb R).
\]
Every connected slice $\Sigma_n$ contains
$g(n)\Lambda^3\in C_0$, so $\Sigma_n\subseteq C_0$. Taking the closure of
their union as in \eqref{eq:model-orbit-closure} yields $C=C_0$, which is
connected.
\end{proof}

\end{document}